\documentclass{article}
\usepackage{amsmath, amssymb, amsthm, graphicx, amscd, amsfonts, amscd}

\usepackage[pagewise, displaymath, mathlines]{lineno}

\newtheorem{theorem}{Theorem}[section]
\newtheorem{lemma}[theorem]{Lemma}

\newtheorem{corollary}[theorem]{Corollary}

\def\bb #1{ {\mathbb #1} }
\def\c #1{ {\mathcal #1} }

\def\u { \underline  }

\def\w #1{ {w^{( #1)} } }
\def\kap #1{ { \kappa^{\u #1} } }
\def\ep { {  \varepsilon } }
\def\L #1{ { \c L_{\kappa, #1} } }

\begin{document}
\title{$L$-functions of symmetric powers of Kloosterman sums \\ (unit root $L$-functions and $p$-adic estimates)}

\author{C. Douglas Haessig\footnote{This work was partially supported by a grant from the Simons Foundation (\#314961 to C. Douglas Haessig).}}

\maketitle

\abstract{The $L$-function of symmetric powers of classical Kloosterman sums is a polynomial whose degree is now known, as well as the complex absolute values of the roots. In this paper, we provide estimates for the $p$-adic absolute values of these roots. Our method is indirect. We first develop a Dwork-type $p$-adic cohomology theory for the two-variable infinite symmetric power $L$-function associated to the Kloosterman family, and then study $p$-adic estimates of the eigenvalues of Frobenius. A continuity argument then provides the desired $p$-adic estimates.}

\tableofcontents

\section{Introduction}

Let $\bb F_q$ be a finite field with $q = p^a$ elements, $p \geq 5$. Associated to each $\bar t \in \overline{\bb F}_q^*$ define the Kloosterman sum
\[
Kl_{\bar t, m} := \sum_{\bar x \in \bb F_{q_{\bar t}^{m}}} \psi \circ Tr_{\bb F_{q_{\bar t}^{m}} / \bb F_q} \left( \bar x + \frac{\bar t}{\bar x} \right) \qquad m = 1, 2, 3, \ldots,
\]
where $deg(\bar t) := [\bb F_q(\bar t): \bb F_q]$, $q_{\bar t} := q^{deg(\bar t)}$, and $\psi$ is a fixed non-trivial additive character on $\bb F_q$. It is well-known that the associated $L$-function is quadratic:
\[
L(Kl_{\bar t}, T) := \exp \left( \sum_{m \geq 1} Kl_{\bar t, m} \frac{T^m}{m} \right)  = (1 - \pi_0(\bar t) T)(1 - \pi_1(\bar t) T),
\]
with roots satisfying  $\pi_0(\bar t) \pi_1(\bar t) = q_{\bar t}$, $|\pi_0(\bar t)|_{\bb C}  = |\pi_1(\bar t)|_{\bb C} = \sqrt{q_{\bar t}}$, and $\pi_0(\bar t)$ is a $p$-adic 1-unit, meaning $|1 - \pi_0(\bar t)|_p < 1$.

For $k$ a positive integer, define the \emph{$k$-th symmetric power $L$-function} of the Kloosterman family by
\[
L(Sym^k Kl, T) := \prod_{\bar t \in |\bb G_m / \bb F_q|} \prod_{m=0}^k \frac{1}{1 - \pi_0(\bar t)^{k-m} \pi_1(\bar t)^m T^{deg(\bar t)}}.
\]
Robba \cite{Robba-SymmetricPowersof-1986} gave the first $p$-adic study of this $L$-function, and showed among other things that it is a polynomial defined over $\bb Z$, and gave a conjectural formula for its degree. Using $\ell$-adic techniques, Fu-Wan \cite{FuWan-$L$-functionssymmetricproducts-2005} proved a corrected version of this conjecture as a special case of their study of the $n$-variable Kloosterman sums. In that paper, motivated by a conjecture of Gouv\^ea-Mazur, they asked whether there exists a uniform quadratic lower bound for the Newton polygon of $L(Sym^k Kl, T)$ independent of $k$. Our first main result of this paper affirmatively answers their question:

\begin{theorem}\label{T: finite sym Kloo estimate}
Let $p \geq 5$. For every positive integer $k$, writing $L(Sym^k Kl, T) = \sum c_m T^m$, then
\begin{equation}\label{E: NP in main theme}
ord_q c_m \geq \left(1 - \frac{1}{p-1} \right) m(m-1) .
\end{equation}
\end{theorem}

A uniform quadratic lower bound is known \cite{Wan-Dimensionvariationof-1998} for the $L$-function of the $k$-th symmetric product of the first relative $\ell$-adic cohomology of the Legendre family of elliptic curves $E_t: x_1^2 = x_1(x_1-1)(x_1 - t)$. This $L$-function, defined analogously to $L(Sym^k Kl, T)$, equals the nontrivial part of the Hecke polynomial associated to the $p$-th Hecke operator $T_{k+2}(p)$ for level 2 acting on cusp forms of weight $k + 2$ (see \cite{Adolphson-$p$-adictheoryof-1976}). For the Kloosterman family, recent work of Yun \cite{MR3305309}, based on conjectures of Evans \cite{MR2310495}, gives an automorphic interpretation for $L(Sym^k Kl, T)$ when $k$ is small. It has also been shown by Fu-Wan \cite{MR2425148} that $L(Sym^k Kl, T)$ is geometric (or motivic) in nature, meaning it equals the local factor at $p$ of the zeta function of a (virtual) scheme of finite type over $\bb Z$.

Our motivation for the following study comes from a related but different direction. In \cite{Dwork-Heckepolynomials-1971}, Dwork first defined the unit root $L$-function of the Legendre family of elliptic curves essentially as follows. Setting $X := \bb A^1 \setminus \{0, 1, H(t) = 0\}$, where $H(t)$ is the Hasse polynomial, the map $f: E_t \mapsto t \in X$ gives a family of ordinary elliptic curves whose first relative $p$-adic \'etale cohomology $R^1 f_* \bb Z_p$ gives a continuous rank one representation $\rho_E: \pi_1^\text{arith}(X) \rightarrow GL_1(\bb Z_p)$. This has the property that, for a closed point $t \in |X / \bb F_q|$, the image of the geometric Frobenius $Frob_t$ is $\rho_E(Frob_t) = \pi_0(t)$, where $\pi_0(t)$ is the unique $p$-adic unit root of the zeta function of the fiber $E_t$. For $k$ a positive integer, define the unit root $L$-function
\[
L(\rho_E^{\otimes k}, T) := \prod_{\bar t \in | X / \bb F_q|} \frac{1}{1 - \pi_0(\bar t)^k T^{deg(\bar t)}}.
\]
For every $k$, meromorphy was shown by Dwork in \cite{Dwork-Heckepolynomials-1971}. It also has a $p$-adic modular interpretation. Define the Fredholm determinant $D(k, T) := det(I - U_p T \mid M_k)$, where  $M_k$ is the space of overconvergent $p$-adic modular forms of level 2 and weight $k$, and $U_p$ is the Atkin operator. Then there is the relation:
\begin{equation}\label{E: p-adic modularity}
L(\rho_E^{\otimes k}, T) = \frac{D(k+2, T)}{D(k, pT)}.
\end{equation}
This allows one to obtain results about modular forms from results about the unit root $L$-function. See \cite{Wan-quickintroductionto-1999} for a detailed exposition; see also \cite{Wan-Dimensionvariationof-1998}. Their special values $L(\rho_E^{\otimes k}, 1)$ are also related to geometric Iwasawa theory as discussed in \cite{MR880957}.

Unit root $L$-functions do not appear to have a nice cohomology theory, however a related $L$-function does appear to have one, which is essentially the main result of this paper. To motivate, notice that using (\ref{E: p-adic modularity}) recursively:
\begin{align*}
D(k, T) &= \prod_{i \geq 0} L_\text{unit}(k - 2 - 2i, q^i T) \\
&= \prod_{\bar t \in | X / \bb F_q|} \prod_{i \geq 1} (1 - \pi_0(\bar t)^{k-2-i} \pi_1(\bar t)^i T^{deg(\bar t)})^{-1}.
\end{align*}
The latter product is similar to the $k-2$ symmetric power except that the product runs over all $i \geq 1$ rather than $1 \leq i \leq k$.  This motivates us to define the \emph{infinite $k$-symmetric power $L$-function} of the Legendre family $E$ by
\[
L(Sym^{\infty, k} E, T) := \prod_{\bar t \in | X / \bb F_q|} \prod_{i \geq 1} (1 - \pi_0(\bar t)^{k-2-i} \pi_1(\bar t)^i T^{deg(\bar t)})^{-1},
\]
and thus the above relation becomes
\begin{equation}\label{E: modular form inf sym}
det(I - U_p T \mid M_k) = L(Sym^{\infty, k-2} E, T).
\end{equation}
(This relation ultimately comes from the Eichler-Selberg trace formula.) We conjecture that there is a $p$-adic cohomology theory for $L(Sym^{\infty, k} E, T)$ that lines up with (\ref{E: modular form inf sym}). Furthermore, this cohomology theory should have a natural dual theory. One possible approach is to take a $p$-adic limit in $k$ of Adolphson's work \cite{Adolphson-$p$-adictheoryof-1976}. This is likely related to $p$-adic modular forms and $p$-adic modular symbols. 

Our main reason to suspect the existence of such a $p$-adic cohomology theory comes from the Kloosterman family studied in this paper. Let $\kappa \in \bb Z_p$, where $\bb Z_p$ denotes the $p$-adic integers. Define the infinite $\kappa$-symmetric power $L$-function (using the roots $\pi_0$ and $\pi_1$ from the Kloosterman sums above)
\[
L(Sym^{\infty, \kappa} Kl, T) := \prod_{\bar t \in | \bb G_m / \bb F_q|} \prod_{m \geq 0} \frac{1}{1 - \pi_0(\bar t)^{\kappa - m} \pi_1(\bar t)^m T^{deg(\bar t)}}.
\]
In this paper, we  develop a $p$-adic cohomology theory $H_\kappa^\bullet$ (Section \ref{S: symmetric power cohom}) for the infinite $\kappa$-symmetric power $L$-function which may be used to meromorphically describe the $L$-function:
\[
L(Sym^{\infty, \kappa} Kl, T) = \frac{det(1 - \bar \beta_{\kappa} T \mid H_\kappa^1)}{det(1 - q \bar \beta_{\kappa} T \mid H^0_\kappa)},
\]
where $\bar \beta_\kappa$ is a completely continuous operator defined on $p$-adic spaces $H^1_\kappa$ and $H^0_\kappa$. We note that these types of $L$-functions are not expected to be rational functions in general. For the Kloosterman family, we will show $H^0_\kappa = 0$ when $\kappa \not=0$, and of dimension one when $\kappa = 0$, and when $\kappa \in \bb Z_p \setminus \bb Z_{\geq 0}$ then $H^1_\kappa$ is infinite dimensional. It is interesting that the case $\kappa \in \bb Z_p \setminus \bb Z_{\geq 0}$ is the easiest to handle, whereas we are unable to compute cohomology when $\kappa$ is a positive integer.

Studying the action of Frobenius on cohomology leads to our third main result:

\begin{theorem}\label{T: main theme}
Let $p \geq 5$. For every $\kappa \in \bb Z_p$, $L(Sym^{\infty, \kappa} Kl, T)$ is $p$-adic entire with $T=1$ a root. Furthermore, writing $L(Sym^{\infty, \kappa} Kl, T) = \sum_{m \geq 0} c_m T^m \in 1+ T\bb Z_p[[T]]$, then $ord_q c_m$ satisfies (\ref{E: NP in main theme}).
\end{theorem}

The proof of Theorem \ref{T: finite sym Kloo estimate} now follows from the following identity: for $k$ a positive integer,
\begin{equation}\label{E: sym Kloo}
L(Sym^k Kl, T) = \frac{L(Sym^{\infty, k} Kl, T)}{L(Sym^{\infty, -(k+2)} Kl, q^{k+1} T)}.
\end{equation}

Some heuristic calculations suggest that there is a chance the lower bound in Theorem \ref{T: main theme} (and thus Theorem \ref{T: finite sym Kloo estimate}) may be improved to simply $m(m-1)$, but we have been unable to prove this, and it may be that it fails for some $m$. We conjecture that the zeros and poles of $L(Sym^{\infty, \kappa} Kl, T)$ are all simple except for possibly finitely many, and that adjoining the collection of zeros and poles to $\bb Q_p$ produces a finite extension field of $\bb Q_p$ (the so-called $p$-adic Riemann hypothesis).

We feel strongly that a relation such as (\ref{E: modular form inf sym}) exists for $L(Sym^{\infty, k} Kl, T)$. As the Kloosterman sums are defined over the totally real field $\bb Q(\zeta_p + \zeta_p^{-1})$, one could look at $p$-adic Hilbert modular forms for such a relation. However, since Kloosterman sums depend on the embedding of the character $\psi$, a relation would involve instead the more complicated $L$-function $L(\otimes_{\sigma} Sym^{\infty, \kappa} Kl_{\sigma}, T)$, where $\sigma$ runs over the real embeddings of $\bb Q(\zeta_p + \zeta_p^{-1})$. Further, running over different symmetric powers for each embedding, the $L$-function 
\[
L(Sym^{\infty, \kappa_1} Kl_{\sigma_1} \otimes \cdots \otimes Sym^{\infty, \kappa_{(p-1)/2}} Kl_{\sigma_{(p-1)/2}}, T)
\]
is likely related to non-parallel weight $(\kappa_1, \ldots, \kappa_{(p-1)/2})$ $p$-adic Hilbert modular forms.

The \emph{Kloosterman unit root $L$-function} is defined by
\begin{equation}\label{E: unit root Kloo}
L_\text{unit}(Kl, \kappa, T) := \prod_{\bar t \in | \bb G_m / \bb F_q|} \frac{1}{1 - \pi_0(\bar t)^\kappa T^{deg(\bar t)}},
\end{equation}
and note that
\begin{equation}\label{E: relate unit to sym}
L_\text{unit}(Kl, \kappa, T) = \frac{ L(Sym^{\infty, \kappa} Kl, T) }{ L(Sym^{\infty, \kappa-2} Kl, qT) },
\end{equation}
and so we have the cohomological description of the unit root $L$-function:
\[
L_\text{unit}(Kl, \kappa, T) = \frac{det(1 - \bar \beta_{\kappa} T \mid H^1_\kappa)}{det(1 - q \bar \beta_{\kappa-2} T \mid H^1_\kappa)} \qquad (\text{when } \kappa \not = 0 \text{ or } 2).
\]
This shows the unit root $L$-function has a root at $T=1$ and a pole at $T = 1/q$. It is unclear whether there are any cancellations among the remaining zeros and poles, however, we expect that there are few if any. We note that while Artin's conjecture does not carry over to geometric $p$-adic representations, it does for infinite symmetric powers over curves. See \cite{MR2364844} for more details.

The $p$-adic cohomology theory developed here is of de Rham type, and may be seen as an extension of Dwork's $p$-adic cohomology theory. We thus expect techniques from Dwork's classical theory may be carried over to this theory. For example, a dual theory seems possible, perhaps giving rise to possible symmetry? Another example is studying the variation of a family of unit root $L$-functions. In joint work with Steven Sperber \cite{Haessig:2015kq}, we examine how the unit roots of a family (of unit root $L$-functions) vary with respect to the parameter by means of establishing a dual theory for infinite symmetric power $L$-functions.

While we have restricted our study to the case of the one-variable Kloosterman family, the cohomology theory developed here may be used for other families, such as those studied in \cite{MR3239170} and \cite{Haessig_2013_Families-of-generalized-Kloosterman-sums}. We hope to say more about their cohomology in a future article.

\section{Relative Bessel cohomology}\label{S: rel bes}

Attached to the Kloosterman family are the relative cohomology spaces $H^0_t(b', b)$ and $H^1_t(b', b)$ defined below. The subscript $t$ is meant to remind us that we will be viewing the Kloosterman family $x + \frac{t}{x}$ with $t$ as a parameter, and thus $H^0_t$ and $H^1_t$ will be modules over function spaces in $t$. In this section, we will recall Dwork's ``Bessel cohomology'' construction \cite[Section 2]{MR0387281} of $H^0_t$ and $H^1_t$ but modified according to \cite{Haessig_2013_Families-of-generalized-Kloosterman-sums}. In particular, we will see that $H^0_t = 0$ and $H^1_t$ is a free module of rank 2 over a certain power series ring $L(b')$. We then study the action of Frobenius on $H^1_t$. In the next section, we will take, in an appropriate sense, the infinite symmetric power of $H^1_t$ and define a cohomology theory on it. 

Throughout this paper we fix a prime $p \geq 5$ (see the beginnings of Sections \ref{SS: Frob} and \ref{SS: Frob-2} for the reason). Next, fix $\pi \in \overline{\bb Q}_p$ satisfying $\pi^{p-1} = -p$, and for convenience, set $\Omega := \bb Q_p(\pi)$. Set $\tilde b := (p-1)/p$. Throughout the following, let $b$ and $b'$ be real numbers satisfying:
\begin{equation}\label{E: b's}
\tilde b \geq b > 1/(p-1) \qquad \text{and} \qquad  b \geq b', \qquad \text{and set } \qquad \ep := b - \frac{1}{p-1}.
\end{equation}

Dwork's theory works best when the spaces are tailored to suit the family in hand. Observe that we may write $\exp \pi (x + \frac{t^q}{x}) = \sum_{n \geq 0, u \in \bb Z} A(n, u) t^n \cdot t^{q m(u)} x^u$, where $m(u) := \max\{ -u, 0 \}$. This guides us in the following definition of the space $K_q(b', b)$ below. Let $\rho \in \bb R$. Define the following spaces:
\begin{align*}
L(b'; \rho) &:= \left\{ \sum_{n \geq 0} A(n) t^n \mid A(n) \in \Omega, ord_p \> A(n) \geq 2b'n + \rho \right\} \\
L(b') &:= \bigcup_{\rho \in \bb R} L(b'; \rho) \\
K_q(b', b; \rho) &:= \left\{ \sum_{n \geq 0, u \in \bb Z} A(n, u) t^n \cdot t^{q m(u)} x^u \mid A(n, u) \in \Omega, ord_p \> A(n, u) \geq 2b'n + b|u| + \rho \right\} \\
K_q(b', b) &:= \bigcup_{\rho \in \bb R} K_q(b', b; \rho),
\end{align*}
where $q = p^a$ with $a \geq 0$. Note that $K_q(b'/q, b)$ is an $L(b'/q)$-module. When $a = 0$, we will often write $K(b', b)$ for $K_1(b', b)$. Define the (twisted) relative boundary operator
\begin{align*}
D_{t^q} :&= x \frac{\partial}{\partial x} + \pi \left( x - \frac{t^q}{x} \right) \\
&= e^{-\pi(x + t^q/x)} \circ x \frac{\partial}{\partial x}  \circ e^{\pi(x + t^q/x)},
\end{align*}
which acts on $K_q(b'/q, b)$, and thus defines the cohomology spaces
\[
H^0_{t^q}(b'/q, b) := ker( D_{t^q} \mid K_q(b'/q, b)) \qquad \text{and} \qquad H^1_{t^q}(b'/q, b) := K_q(b'/q, b) / D_{t^q} K_q(b'/q, b).
\]
Define
\begin{align*}
V_q(b', b; \rho) &:= \left(\Omega[[t]] + \Omega[[t]] \cdot \frac{t^q}{x} \right) \cap K_q(b', b; \rho) \\
V_q(b', b) &:= \bigcup_{\rho \in \bb R} V_q(b', b; \rho).
\end{align*}
When $a = 0$, we will write $V(b', b)$ for $V_1(b', b)$.

\begin{theorem}[Theorem 2.1 of \cite{MR0387281}]\label{T: rel dwork decomp}
We have
\[
H^0_{t^q}(b'/q, b) = 0 \qquad \text{and} \qquad H^1_{t^q}(b'/q, b) \cong V_q(b'/q, b).
\]
More specifically, 
\[
K_q(b'/q, b; 0) = V_q(b'/q, b; 0) \oplus D_{t^q} K_q(b'/q, b; \ep).
\]
\end{theorem}

\begin{proof}
This is \cite[Theorem 2.1]{MR0387281} but modified slightly to suit the spaces defined above. Dwork does not explicitly point out that $\text{ker } D_{t^q} = 0$, however, this follows immediately from \cite[Lemma 2.5]{MR0387281}.
\end{proof}

As a consequence, we from now on will identify the first cohomology group $H_{t^q}^1(b' / q, b)$ with $V_q(b'/q, b)$, a free $L(b'/q)$-module of rank two with basis $\{ 1, \pi t^q / x \}$. We now study the action of Frobenius on this space.

\bigskip \noindent {\bf Relative Frobenius.} Dwork's Frobenius, denoted $\alpha_a$ below, is essentially the geometric Frobenius map $\psi_x: x \mapsto x^{1/q}$ but twisted in a similar manner to that of $D_t$. It is defined as follows. First, define Dwork's splitting function
\[
\theta(z) := \exp \pi (z - z^p) = \sum_{i \geq 0} \theta_i z^i,
\]
where it is well-known that $ord_p \> \theta_i \geq (p-1) i / p^2$. The splitting function gives a $p$-adic analytic representation of an additive character on $\overline{\bb F}_q$: specifically, $\theta(1)$ is a primitive $p$-th root of unity thanks to the oddness of $p$-adic analysis, and for $\bar z \in \bb F_q$ with $q = p^a$, and $\hat z$ the Teichm\"uller lift of $\bar z$,
\[
\theta(1)^{Tr_{\bb F_q / \bb F_p}( \bar z)} = \theta(\hat z) \theta(\hat z^p) \cdots \theta(\hat z^{p^{a-1}}).
\]
See \cite[Prop. 6.2]{MR1274045} for discussion of this splitting function. We now consider the $p$-adic analytic analogue of $\theta(1)^{Tr_{\bb F_q / \bb F_p}( x + \frac{t}{x})} = \theta(1)^{Tr(x)} \theta(1)^{Tr(t/x)}$: for each $m \geq 1$, define
\begin{align*}
F(t, x) &:= \theta(x) \theta(t/x) \\
F_m(t, x) &:= \prod_{i=0}^{m-1} F(t^{p^i}, x^{p^i}).
\end{align*}
Define the operator $\psi_x: \sum A_u x^u \mapsto \sum A_{pu} x^u$, and observe that $\psi_x: K(b', b; 0) \rightarrow K(b', pb; 0)$. Dwork's Frobenius is defined by
\begin{align*}
\alpha_1(t) :&= \psi_x \circ F(t, x) \\
&= e^{-\pi(x + t/x)} \circ \psi_x  \circ e^{\pi(x + t/x)},
\end{align*}
and for $m \geq 1$,
\begin{align*}
\alpha_m(t) :&= \psi^m_x \circ F_m(t, x) \\
&= \alpha_1(t^{p^{m-1}}) \circ \cdots \alpha_1(t^p) \circ \alpha_1(t).
\end{align*}
In this definition $F(t, x)$ acts via multiplication. Now, since $F(t, x) \in K(\tilde b / p, \tilde b / p; 0)$ we see that $F(t^{p^i}, x^{p^i}) \in K(\tilde b/ p^{i+1}, \tilde b / p^{i+1}; 0 )$, and thus we have the well-defined map $\alpha_m(t): K(b', b; 0) \rightarrow K_{p^m}(b' / p^m, b; 0)$. Furthermore, by construction, 
\[
p^m D_{t^{p^m}} \circ \alpha_m = \alpha_m \circ D_t,
\]
and so $\alpha_m$ induces a map on relative cohomology $\bar \alpha_m(t): H_{t}^1(b', b) \rightarrow H_{t^{p^m}}^1(b'/p^m, b)$. 

The next theorem provides details on the entries of the matrix of $\bar \alpha_m(t)$ with respect to the bases $\{1, \pi t/x\}$ and $\{1, \pi t^{p^m} / x\}$. These entries will consist of power series in $t$ whose specific growth conditions will help us give a well-defined Frobenius map on the infinite symmetric power spaces of $H^1_t$ given in Section \ref{SS: Frob}.

\begin{theorem}\cite[Theorem 2.2 and Lemma 3.2]{MR0387281}\label{T: Dwork rel cohom}
With $\{1, \pi t/x\}$ and $\{1, \pi t^{p^m} / x\}$ as bases of $H_t^1(\tilde b, \tilde b)$ and $H_{t^{p^m}}^1(\tilde b/p^m, \tilde b)$, respectively, the relative Frobenius $\bar \alpha_m$ satisfies
\[
\bar \alpha_m(t) (1) = A_{m, 1}(t) + A_{m, 3}(t) \frac{\pi t^{p^m}}{x} \qquad \text{and} \qquad \bar \alpha_m(t)(\frac{\pi t}{x}) = A_{m, 2}(t) + A_{m, 4}(t) \frac{\pi t^{p^m}}{x},
\]
where
\begin{equation}
\begin{split}
A_{m, 1} &\in L(\tilde b/p^m ; 0) \\
A_{m, 2} &\in L(\tilde b / p^m; \frac{1}{p-1} - \frac{\tilde b}{p})
\end{split}
\qquad  \qquad 
\begin{split}
A_{m, 3} &\in L(\tilde b/p^m; \tilde b - \frac{1}{p-1}) \\
A_{m, 4} &\in L(\tilde b/p^m; \tilde b - \frac{\tilde b}{p})
\end{split}
\end{equation}
Furthermore,
\[
A_{m, 1}(0) = 1 \quad A_{m, 2}(0) = 0 \quad A_{m, 3}(0) \not= 0 \quad A_{m, 4}(0) = p^m.
\]

\end{theorem}

\section{$Sym^{\infty, \kappa}$-cohomology}\label{S: symmetric power cohom}

We now consider the meaning of taking the infinite symmetric power of the space $H^1_t$ from the previous section and how to construct a cohomology on it. Later, in Section \ref{SS: Frob} we will define and study a Frobenius map on these spaces. For now, we continue with the same restrictions (\ref{E: b's}) on $p$, $b$, and $b'$ from the beginning of Section \ref{S: rel bes}.

Let $\kappa \in \bb Z_p$. Denote by $\Omega[[t, w]]$ the formal power series ring in the $t$ and $w$. Intuitively, we will think of $\Omega[[t, w]]$ as the $\kappa$-symmetric power of the space $H_{t}^1$ with basis $\{1, \pi t/x \}$ by writing $w$ for the basis vector $\pi t / x$, and 1 for the other basis vector 1. Then we should intuitively view the monomial $w^m \in \Omega[[t, w]]$ as the $\kappa$-symmetric power vector $1^{\kappa - m} w^m$. This intuition will guide us through our definitions below.

In order to avoid certain denominators in $\kappa$, we will use a normalization $w^{(m)}$ of the monomial $w^m$. This is defined as follows. First, define the falling factorial $\kap{m}$:  for $\kappa \in \bb Z_p \setminus \bb Z_{\geq 0}$ and $m \geq 0$, the definition is conventional: 
\[
\kap{m} := \kappa( \kappa - 1) \cdots (\kappa - m +1),
\]
where $\kap{0} := 1$. For $\kappa = k \in \bb Z_{\geq 0}$, define
\[
k^{\u m} := 
\begin{cases}
1 & \text{if } m = 0 \\
k (k-1) \cdots (k-m+1) & \text{if } 0 < m \leq k \\
k(k-1) \cdots 2 \cdot 1 \cdot \hat 0 \cdot (-1) \cdot (-2) \cdots (k - m +1) & \text{if } m \geq k+1,
\end{cases}
\]
where $\hat 0$ means it is counted in the $m$ terms of the product, but omitted from the actual product. For example, $k^{\u k} = k^{\u{k+1}} = k!$ and $k^{\u{k+2}} = k! \cdot (-1)$; and $0^{\u 0} := 1$, $0^{\u 1} := \hat 0 = 1$, and $0^{\u 2} := -1$. We now define $\w{m} := \kap{m} w^m$. A more natural choice of basis (due to the definition of $[\bar \alpha_a]_\kappa$ below) is $\binom{\kappa}{m} w^m$; however, this leads to a non-integral cohomology theory, which in turn makes obtaining $p$-adic estimates difficult. As in the previous section, we need to limit our space to one with specific growth conditions. That is, define the subspaces of $\Omega[[t, w]]$:
\begin{align*}
S(b', \ep; \rho) &:= \left\{ \sum_{n, m \geq 0} A(n, m) t^n  \w{m} \mid A(n, m) \in \Omega, ord_p \> A(n, m) \geq 2b' n + \ep m + \rho \right\} \\
S(b', \ep) &:= \bigcup_{\rho \in \bb R} S(b', b; \rho).
\end{align*}
We now define a boundary map on this space by taking the infinite symmetric power of the Gauss-Manin connection. Recall, the Gauss-Manin connection for the Kloosterman family is the following. With
\begin{align*}
\partial :&= t \frac{\partial}{\partial t} + \frac{\pi t}{x} \\
&= e^{- \pi(x + t/x)} \circ t \frac{\partial}{\partial t} \circ e^{\pi(x + t/x)},
\end{align*}
observe that $\partial$ and $D_t$ commute, and thus $\partial$ acts on the relative cohomology space $H_t^1$. Its action may be computed explicitly: since
\begin{equation}\label{E: partials}
\partial(1) = \frac{\pi t}{x} \qquad \text{and} \qquad \partial( \frac{\pi t}{x} ) = \pi^2 t,
\end{equation}
we see that on $H_{t}^1$ with basis $\{1,  \pi t/x \}$ the map $\partial$ takes the matrix form (acting on row vectors)
\begin{equation}\label{E: diff of partial}
\partial = t \frac{d}{dt} + H, \qquad \text{where} \qquad H :=
\left(
\begin{array}{cc}
0  & 1   \\
\pi^2 t  & 0 
\end{array}
\right).
\end{equation}
This is the Gauss-Manin connection of the Kloosterman family. (The solutions of this family at 0 and $\infty$ are studied in detail by Dwork \cite{MR0387281}.) 

As mentioned earlier, if we intuitively set $w = \pi t/x$ then we may write (\ref{E: partials}) as $\partial(1) = w$ and $\partial(w) = \pi^2 t$. Viewing $w^m$ as  $1^{\kappa - m} w^m$ suggests we may define the $\kappa$-symmetric power of $\partial$ on $\Omega[[t, w]]$ by the product rule:
\begin{align*}
\partial_\kappa( w^m) &= \text{`` $\partial( 1^{\kappa - m} w^m ) $ ''} \\
&= (\kappa - m) \partial(1) w^m + m w^{m-1} \partial(w) \\
&= (\kappa - m) w^{m+1} + m \pi^2 t w^{m-1}.
\end{align*}
Adding in $t \frac{d}{dt}$ finishes the definition:
\[
\partial_\kappa(t^n w^m) := n t^n w^m + (\kappa-m) t^n w^{m+1} + m \pi^2 t^{n+1} w^{m-1}.
\]
Since we are using the normalized basis $\{w^{(m)}\}$, which modifies the coefficients, we restate the definition. Define the boundary map $\partial_\kappa: S(b', \ep) \rightarrow S(b', \ep)$ as follows. Let $m \geq 0$. For $\kappa \in \bb Z_p \setminus \bb Z_{\geq 0}$, define
\begin{equation}\label{E: boundary}
\partial_\kappa(t^n \w{m}) := n t^n \w{m} +  t^n \w{m+1} + m (\kappa- m +1) \pi^2 t^{n+1} \w{m-1}.
\end{equation}
Note that $\partial_\kappa S(b', \ep; \ep) \subset S(b', \ep; 0)$. For $\kappa = k \in \bb Z_{\geq 0}$, define $\partial_k$ as follows. For $0 \leq m \leq k-1$ or $m \geq k+2$, then $\partial_k(t^n \w{m})$ is defined by (\ref{E: boundary}). When $m = k$ or $m = k+1$, define
\begin{align*}
\partial_k(t^n \w{k}) &:= n t^n \w{k} + k \pi^2 t^{n+1} \w{k-1} \\
\partial_k(t^n \w{k+1}) &:= n t^n \w{k+1} + t^n \w{k+2} + (k+1) \pi^2 t^{n+1} \w{k}.
\end{align*}
Define the cohomology spaces
\[
H_{\kappa}^0(S(b', \ep)) := ker(\partial_\kappa) \quad \text{and} \quad H_{\kappa}^1(S(b', \ep)) := S(b', \ep) / \partial_\kappa S(b', \ep).
\]

It is useful to have a matrix version of $\partial_\kappa$, especially when showing $H_\kappa^0 = 0$ for almost all values of $\kappa$ (Theorem \ref{T: 0-homology}). To do this, we begin by writing the basis vectors $1$ and $\pi t/x$ in row vector form $(1,0)$ and $(0,1)$ so that $(1,0)H = (0,1)$ and $(0,1)H = (\pi^2 t, 0)$. We will abuse notation and write these as $H(1) = w$ and $H(\pi t/x) = \pi^2 t$. This defines the derivation (or Leibniz rule):
\begin{align}\label{E: Lie derivative}
\L{H}(w^m) &:= (\kappa - m) H(1) w^m + m w^{m-1} H(w)  \notag \\
&= (\kappa - m) w^{m+1} + m \pi^2 t w^{m-1}.
\end{align}
In terms of the normalized basis this means: for $\kappa \in \bb Z_p \setminus \bb Z_{\geq 0}$,
\[
\L{H}(w^{(m)}) = t^n \w{m+1} + m (\kappa- m +1) \pi^2 t^{n+1} \w{m-1}.
\]
Thus, $\partial_\kappa$ on $S(b', \ep)$ takes the form 
\[
\partial_\kappa = t \frac{d}{dt} + \L{H},
\]
an infinite differential system, where the matrix of $\L{H}$, acting on row vectors, takes the following form. For $\kappa \in \bb Z_p \setminus \bb Z_{\geq 0}$:
\[
\text{matrix of $\L{H}$} = 
\left(
\begin{array}{cccccc}
 0  & 1  &    &  &  & \\
 \kappa \pi^2 t & 0  & 1  &  & & \\
  & 2 (\kappa-1) \pi^2 t  & 0 & 1 &  &\\
   &  &  3 (\kappa - 2) \pi^2 t & 0 & & \ddots \\
   &  &  & \ddots & &
\end{array}
\right).
\]
For $\kappa = k \in \bb Z_{>0}$, the matrix takes the form
\[
\left(
\begin{array}{cccccccccccc}
0               & 1                   &                     &        &   &                     &                       &                          &                        &        &   &  \\
1 \cdot k \pi^2 t & 0                   & 1                   &        &   &                     &                       &                          &                        &        &   &  \\
                & 2 \cdot (k-1) \pi^2 t & 0                   & 1      &   &                     &                       &                          &                        &        &   &  \\
                &                     & 3 \cdot (k-2) \pi^2 t & 0      & \quad1 &                     &                       &                          &                        &        &   &  \\
                &                     &                     & \ddots &   & \ddots              &                      &                          &        \text{Note}                &        &   &  \\
                &                     &                     &        &   & (k-1) \cdot 2 \pi^2 t & 0                     & 1                        &        \downarrow                &        &   &  \\
                &                     &                     &        &   &                     & k \cdot 1 \cdot \pi^2 t & 0                        & 0                      &        &   &  \\
                &                     &                     &        &   &                     &                       & (k+1) \cdot \hat 0 \pi^2 t & 0                      & 1      &   &  \\
                &                     &                     &        &   &                     &                       &                          & (k+2) \cdot (-1) \pi^2 t & 0      & 1 &  \\
                &                     &                     &        &   &                     &                       &                          &                        & \ddots &   & 
\end{array}
\right).
\]
where $\hat 0$ means it is omitted. Note the zero in the $(k+2)$ column. For $\kappa = 0$, we have
\[
\text{matrix of $\c L_{0, H}$} = 
\left(
\begin{array}{cccccc}
0                        & 0                      &    0    &   &  & \\
                        & 2 \cdot \hat 0 \pi^2 t & 0                      & 1      &   &  \\
         &                          & 2 \cdot (-1) \pi^2 t & 0      & 1 &  \\
                &                     &                     &     \ddots   &   &     \ddots               
\end{array}
\right)
\]

\subsection{$H^0_\kappa$}

It follows almost immediately from this matrix description above that $H_\kappa^0 = 0$ for all $\kappa \in \bb Z_p \setminus \bb Z_{\geq 0}$ (see Theorem \ref{T: 0-homology} below). This also holds when $\kappa$ is a positive integer, but it is not immediate. When $\kappa$ a positive integer, then the infinite differential system breaks up into two systems, one finite and one infinite. The finite system will be precisely the $k$-th symmetric power of the Gauss-Manin connection $\partial = t \frac{d}{dt} + H$. A study of solutions of this latter system uses Robba's work on the symmetric powers of the Kloosterman family, and ultimately a deep result of Dwork's. That this finite system exists is interesting, and is likely analogously related to how classical modular forms of weight $k$ sit inside the space of $p$-adic modular forms of weight $k$. We also note that in the next section, when studying $H_\kappa^1$, we will encounter the same issue, that things are easier when $\kappa$ is not a positive integer.

\begin{theorem}\label{T: 0-homology}
Let $\kappa \in \bb Z_p$. Then
\[
H_{\kappa}^0(S(b', \ep)) = 
\begin{cases}
0 & \text{if } \kappa \in \bb Z_p \setminus \{ 0 \} \\
\Omega & \text{if } \kappa = 0.
\end{cases}
\]
\end{theorem}

\begin{proof}
We first suppose $\kappa \in \bb Z_p \setminus \bb Z_{\geq 0}$.  Let $\xi \in S(b', \ep)$ be such that $\partial_\kappa \xi = 0$. Writing $\xi = \sum_{n \geq 0} \xi_n \w{n}$, then $\partial_\kappa \xi = 0$ takes the form
\begin{alignat}{3}\label{E: kernel of partial}
  t \xi_0' & & &+ \kappa \pi^2 t \xi_1 &  &= 0 \notag\\
  t \xi_1' &+ \xi_0 &  &+ 2(\kappa -1)\pi^2 t \xi_2 &  &= 0 \\
  t \xi_2' &+ \xi_1 & &+ 3(\kappa - 2)\pi^2 t \xi_3 & &= 0 \notag \\ 
  &\vdots \qquad  \qquad \vdots& & & \vdots   & \notag
\end{alignat}
We will show $t^m \mid \xi_n$ for every $m \geq 1$ and $n \geq 0$ using induction, and thus $\xi = 0$. Observe that the second equation of (\ref{E: kernel of partial}) implies $t \mid \xi_0$, and the third equation implies $t \mid \xi_1$, and so forth: $t \mid \xi_n$ for $n \geq 0$. Next, suppose $t^m \mid \xi_n$ for every $n \geq 0$. The first equation of (\ref{E: kernel of partial}) then implies that $t^{m+1} \mid \xi_0$. Using this, the second equation of (\ref{E: kernel of partial}) shows $t^{m+1} \mid \xi_1$. Continuing, we see that $t^{m+1} \mid \xi_n$ for every $n \geq 0$. Hence, $\xi = 0$ as desired.

Suppose now that $\kappa \in \bb Z_{\geq 1}$. For convenience, set $k := \kappa$. We first observe that $\partial_\kappa \xi = 0$ breaks up into two systems of the form
\begin{alignat}{3}\label{E: kernel of partial-2a}
  t \xi_0' &  & &+ k \pi^2 t \xi_1 & &= 0 \notag\\
  t \xi_1' &+  \xi_0 & &+ 2(k-1)\pi^2 t \xi_2 & &= 0 \\
  t \xi_2' &+  \xi_1 & &+ 3(k-2)\pi^2 t \xi_3 & &= 0 \notag \\ 
  &\vdots \qquad  \qquad \vdots& & & \vdots   & \notag \\
  t \xi_{k+1}' &+ \xi_{k} & &+ (k+1)\pi^2 t \xi_{k+2} & &= 0 \notag   
\end{alignat}
and
\begin{alignat}{3}\label{E: kernel of partial-2b}
  t \xi_{k+2}' &   & &+ (k+2) \cdot (-1)\pi^2 t \xi_{k+3} & &= 0 \\
  t \xi_{k+3}' &+ \xi_{k+2} & &+ (k+3) \cdot (-2)\pi^2 t \xi_{k+4} & &= 0 \notag \\ 
  t \xi_{k+4}' &+ \xi_{k+3} & &+ (k+4) \cdot (-3)\pi^2 t \xi_{k+5} & &= 0 \notag \\ 
  &\vdots \qquad  \qquad \vdots& & & \vdots   & \notag
\end{alignat}
As (\ref{E: kernel of partial-2b}) is of a form essentially identical to (\ref{E: kernel of partial}), a similar argument shows $\xi_m = 0$ for every $m \geq k+2$. Thus, (\ref{E: kernel of partial-2a}) takes the form of the (finite dimensional) differential system $t \frac{d}{dt} + H_k$, where $H_k$ is the $(k+1) \times (k+1)$ matrix
\[
H_k = 
\left(
\begin{array}{cccccc}
 0  & 1  &    &  &    & \\
 k \pi^2 t & 0  & 1    & & & \\
  & 2(k-1) \pi^2 t  & 0 & 1   & &\\
   &  &  3(k-2)\pi^2 t & 0 &  \ddots &  \\
   &  &  & \ddots & & 1  \\
   & & & & k \pi^2 t & 0
\end{array}
\right)
\]
This is precisely the $k$-th symmetric power of the differential system (\ref{E: diff of partial}). Using a deep result of Dwork's, Robba \cite[p.202]{Robba-SymmetricPowersof-1986} shows this system has no overconvergent solutions.

Lastly, set $\kappa = 0$, then $\partial_\kappa \xi = 0$ takes the form
\begin{alignat}{3}
  t \xi_0' &  & &+  \pi^2 t \xi_1 & &= 0 \notag\\
  t \xi_1' & & &+ 2(-1)\pi^2 t \xi_2 & &= 0 \\
  t \xi_2' &+  \xi_1 & &+ 3(-2)\pi^2 t \xi_3 & &= 0 \notag \\ 
  t \xi_3' &+  \xi_2 & &+ 4(-3)\pi^2 t \xi_4 & &= 0 \notag \\ 
  &\vdots \qquad  \qquad \vdots& & & \vdots   & \notag    
\end{alignat}
Ignoring the first equation, this system is of a similar form to (\ref{E: kernel of partial}), and so a similar argument shows $\xi_m = 0$ for every $m \geq 1$. Consequently, the first equation now becomes $\xi_0' = 0$, and so $\xi_0$ is a constant, finishing the proof.
\end{proof}

\subsection{$H^1_\kappa$ with $\kappa \in \bb Z_p \setminus \bb Z_{\geq 0}$}

Next we study the first cohomology group $H_\kappa^1$, assuming throughout this section that $\kappa \in \bb Z_p \setminus \bb Z_{\geq 0}$. We have been unable so far to handle the case when $\kappa$ is a positive integer due to obstacles trying to obtain a result similar to Lemma \ref{L: Dwork decomp H}. Our goal of this section is to prove a similar decomposition result for $S(b', \ep)$ to that of $K(b', b)$ in Theorem \ref{T: rel dwork decomp}. 

We start by first noticing that $L(b'; \rho)$ and $L(b')$ sit naturally in $S(b', \ep)$ by sending $\xi \mapsto \xi$, the series with no $\w{m}$ terms. We will denote by $R(b')$ and $R(b'; \rho)$ the images of these spaces in $S(b', \ep)$. We will show that $H_{\kappa}^1(S(b', \ep))  \cong R(b')$. 

\begin{lemma}
Let $\kappa \in \bb Z_p \setminus \bb Z_{\geq 0}$. Then $R(b') \cap \partial_\kappa S(b', \ep) = \{0\}$.
\end{lemma}

\begin{proof}
Let $\eta \in R(b') \cap \partial_\kappa S(b', \ep)$. Write $\eta = \sum_{n \geq 0} a_n t^n$, and let $\xi \in S(b', \ep; \rho)$ such that $\partial_\kappa \xi = \eta$. Write $\xi = \sum_{n, m \geq 0} A(n, m) t^n \w{m}$ with $ord_p A(n, m) \geq 2b' n + \ep m + \rho$ for some $\rho \in \bb R$. Set $\xi_m := \sum_{n \geq 0} A(n, m) t^n$. Using this, $\xi = \sum_{m \geq 0} \xi_m \w{m}$. Using the $\{ \w{m} \}_{m \geq 0}$ as a basis of $S(b', \ep)$ as an $L(b')$-module, $\xi$ takes the vector form $(\xi_0, \xi_1, \ldots)$, and $\eta$ takes the form $( \sum a_n t^n, 0, 0, \ldots)$. Writing $\partial_\kappa = t \frac{d}{dt} + \c L_{\kappa, H}$ then $\partial_\kappa \xi = \eta$ is equivalent to the system
\[
\begin{cases}
\sum_{n \geq 0} a_n t^n &= t \xi_0' \quad \qquad \qquad + \kappa \pi^2 t \xi_1 \\
0 &= t \xi_1' +  \xi_0 \quad \qquad + 2 (\kappa-1)\pi^2 t \xi_2 \\
0 &= t \xi_2' +  \xi_1 \quad \qquad + 3 (\kappa-2) \pi^2 t \xi_3 \\
&\vdots
\end{cases}
\]
We will show by induction that $t^n \mid \eta$ for every $n$. Observe that the right-hand side of the first equation of the system is divisible by $t$, and thus $a_0 = 0$, or equivalently, $t$ divides $\eta$. Similarly, the second equation shows $t$ divides $\xi_0$. Continuing, we get that $t$ divides every $\xi_i$ for all $i \geq 0$.

Next, as $\kappa \pi^2 t \xi_1$ in the first equation is now divisible by $t^2$, the coefficient of $t$ in $\xi_0$ must equal $a_1$. Using this, from the second equation and the fact that $2(\kappa-1) \pi^2 t \xi_2$ is divisible by $t^2$, the coefficient of $t$ of $\xi_1$ must equal $- a_1$. The same argument using the third equation shows the coefficient of $t$ in $\xi_2$ equals $a_1$. Continuing the argument, we must have $\xi_m = (-1)^m a_1 t + O(t^2)$ for every $m \geq 0$, and so $A(1, m) = (-1)^m  a_1$. As $ord_p A(1, m) \geq 2b' + \ep m + \rho$ for every $m$, it must be that $a_1 = 0$.

Assume now that $t^n$ divides $\eta$ and every $\xi_m$. We proceed with an identical argument as above to show $a_n = 0$. As $\kappa \pi^2 t \xi_1$ in the first equation is now divisible by $t^{n+1}$, the coefficient of $t^n$ in $\xi_0$ must equal $(1/n) a_n$. Using this, from the second equation and the fact that $2 (\kappa-1) \pi^2 t \xi_2$ is divisible by $t^{n+1}$, the coefficient of $t^n$ of $\xi_1$ must equal $- (1/n^2) a_n$. The same argument using the third equation shows the coefficient of $t^n$ in $\xi_2$ equals $(1/n^3) a_n$. Continuing the argument, we must have $\xi_m = (-1)^m (1/n^{m+1}) a_n t^n + O(t^{n+1})$ for every $m \geq 0$. This means $ord_p( \frac{1}{n^{m+1}} a_n) \geq 2b' n + \ep m + \rho$ for every $m$, which is impossible unless $a_n = 0$. Thus, $\eta = 0$.
\end{proof}

\begin{lemma}\label{L: Dwork decomp H}
Let $\kappa \in \bb Z_p \setminus \bb Z_{\geq 0}$. Then
\begin{equation}\label{E: Dwork decomp H}
S(b', \ep; 0) = R(b'; 0) \oplus \L{H}S(b', \ep; \ep).
\end{equation}
\end{lemma}

\begin{proof}
We first observe that the righthand side of (\ref{E: Dwork decomp H}) is contained in the left. We now show the reverse direction. First, observe that 
\[
\w{m+1}  = - m (\kappa- m +1) \pi^2 t \w{m-1} + \L{H}\w{m}.
\]
Using this recursively, it follows that: 
\begin{align*}
&\text{for $m \geq 0$:} \qquad  \w{2m+1} = \L{H} \left( \sum_{l=0}^m \zeta_{2m+1, l} \pi^{2l} t^l \w{2(m-l)} \right) \\
&\text{for $m \geq 1$:} \qquad  \w{2m} = \eta_{2m, 0} \pi^{2m} t^m + \L{H} \left( \sum_{l=0}^{m-1} \zeta_{2m, l} \pi^{2l} t^l \w{2(m-l)-1} \right),
\end{align*}
where $\eta_0$ and $\zeta_l$ are elements in $\bb Z_p$. The result follows easily from this. For future reference we record:
\begin{align*}
\eta_{2m, 0} & = 2^{2m} (\kappa/2)_m (-1/2)_m \\
\zeta_{2m, l} &:= 2^{2l} (m- \frac{1}{2})_l (-\frac{\kappa}{2} +m+1)_l \\
\zeta_{2m+1, l} &:= 2^{2l} (m)_l (\frac{\kappa + 1}{2} + m)_l.
\end{align*}
\end{proof}

\begin{theorem}
Let $\kappa \in \bb Z_p \setminus \bb Z_{\geq 0}$. Then
\begin{equation}\label{E: Dwork decomp}
S(b', \ep; 0) = R(b'; 0) \oplus \partial_\kappa S(b', \ep; \ep).
\end{equation}
Hence, $H_{\kappa}^1(S(b', \ep))  \cong R(b')$.
\end{theorem}

\begin{proof}
First, observe that the righthand side of (\ref{E: Dwork decomp}) is contained in the left. Let $\xi \in S(b', \ep; 0)$. By Lemma \ref{L: Dwork decomp H},  there exists $\eta_0 \in R(b'; 0)$ and $\zeta_0 \in S(b', \ep; \ep)$ such that $\xi = \eta_0 + \L{H}\zeta_0$. Setting $\xi_1 := - t\frac{d}{dt} \zeta_0$, then $\xi = \eta_0 + \partial_\kappa \zeta_0 + \xi_1$. As $\xi_1 \in S(b', \ep; \ep)$, which is an increase by $\ep$ in valuation, we may repeat this process to obtain $\xi = \sum_{m \geq 0} \eta_m + \partial_\kappa \sum_{m \geq 0} \zeta_m$, with $\eta_m \in S(b', \ep; \ep m)$ and $\zeta_m \in S(b', \ep; \ep(m+1))$. 
\end{proof}

\subsection{Frobenius}\label{SS: Frob}

Now that we have completed our study of the cohomology spaces $H_{\kappa}^0$ and $H_\kappa^1$, we move on to the study of the action of Frobenius. First we need to make sense of taking the (infinite) $\kappa$-symmetric power of the relative Frobenius $\alpha_a$ on $S(b', \ep)$. As before, intuitively, if we view $w^m$ as the $\kappa$-symmetric power vector $1^{\kappa - m} w^m$, then the $\kappa$-symmetric power of $\alpha_a$ should act on this by $(\alpha_a(1))^{\kappa - m} \cdot \alpha_a(w)$. Thus, we need to make sure $(\alpha_a(1))^{\kappa - m}$ is well-defined. Furthermore, we need to show such a map is well-defined on the space $S(b', \ep)$. Unfortunately, it is not, unless we make one further restriction on $b$. This is done now. 

We first show that it makes sense to write $(\alpha_a(1))^{\kappa - m}$. By Theorem \ref{T: Dwork rel cohom}, $\bar \alpha_a(1) \in L(b'/q; 0) + L(b'/q; \ep) \frac{\pi t^q}{x}$. Furthermore, $\bar \alpha_a(1) = 1 + \eta + \zeta \frac{\pi t^q}{x}$ with $\eta \in L(b'/q; 0)$, $t \mid \eta$, and $\zeta \in L(b'/q; \ep)$. Define $\Upsilon_q: H_{t^q}^1(K_q(b'/q, b)) \rightarrow \Omega[[t, w]]$  by sending $\zeta + \xi \frac{\pi t^q}{x} \longmapsto \zeta + \xi w$.  Thus, for $\tau \in \bb Z_p$, $( \Upsilon_q \circ \bar \alpha_a(1) )^\tau$ is a well-defined element of $\Omega[[t, w]]$. 

We now define the $\kappa$-symmetric power of $\alpha_a$ as follows: define $[\bar \alpha_a]_\kappa: S(b', \ep) \rightarrow S(b' / q, \ep)$ by linearly extending over $\Omega[[t]]$ the action
\[
[\bar \alpha_a]_\kappa( \w{m} ) := \kappa^{\u{m}} \cdot ( \Upsilon_q \circ \bar \alpha_a(1) )^{\kappa - m} (\Upsilon_q \circ \bar \alpha_a \frac{\pi t}{x} )^m.
\]
Due to the weighted basis $\{ w^{(m)} \}$, it is not immediate that this is well-defined, and in fact may not be without further conditions on $b$. We state this now. Recall, $\tilde b = (p-1)/p$. Assume that
\begin{equation}\label{E: b restrictions}
\tilde b - \frac{1}{p-1} \geq b  > \frac{1}{p-1} \qquad \text{and} \qquad b \geq b'.
\end{equation}
Note, $p \geq 5$ implies $\tilde b > 2/(p-1)$ and so (\ref{E: b restrictions}) is not vacuous. 

\begin{lemma}\label{L: well-defined}
Assuming (\ref{E: b restrictions}) then $[\bar \alpha_a]_\kappa: S(b', \ep) \rightarrow S(b' / q, \ep)$ is well-defined.
\end{lemma}

\begin{proof}
We now make use of the estimates from Theorem \ref{T: Dwork rel cohom}. Using the notation from that theorem, set $\tilde A_{a, 3} := A_{a, 3} / A_{a, 1}$. Then
\begin{align*}
[\bar \alpha_a]_\kappa(w^{(m)}) &= \kappa^{\u m} ( A_{a, 1} + A_{a, 3} w)^{\kappa - m} (A_{a, 2} + A_{a, 4} w)^m \\
&= \kappa^{\u m} A_{a, 1}^{\kappa - m} ( 1 + \tilde A_{a, 3} w)^{\kappa - m} (A_{a, 2} + A_{a, 4} w)^m \\
&= \kappa^{\u m} A_{a, 1}^{\kappa - m}  \left( \sum_{l=0}^\infty \binom{\kappa-m}{l} \tilde A_{a, 3}^l w^l \right) \left( \sum_{n=0}^m \binom{m}{n} A_{a, 2}^{m-n} A_{a, 4}^n w^n \right) \\
&= \sum_{l \geq 0, 0 \leq n \leq m} C(l, n) w^{(l+n)},
\end{align*}
where
\[
C(l, n) := A_{a, 1}^{\kappa - m} \tilde A_{a, 3}^l A_{a, 2}^{m-n} A_{a, 4}^n \binom{m}{n} \binom{\kappa-m}{l} \frac{\kappa^{\u m}}{\kappa^{\u {l+n}}}.
\]
Observe that
\[
\binom{m}{n} \binom{\kappa-m}{l} \frac{\kappa^{\u m}}{\kappa^{\u {l+n}}} \in \frac{1}{l!} \bb Z_p,
\]
and so its $p$-adic valuation is at least $- l / (p-1)$. Using the estimates from Theorem \ref{T: Dwork rel cohom}, we see that $A_{a, 1}^{\kappa - m} \tilde A_{a, 3}^l A_{a, 2}^{m-n} A_{a, 4}^n \in L(b'/q; c)$, where
\[
c := l ( \tilde b - \frac{1}{p-1}) + (m-n)(\frac{1}{p-1} - \frac{\tilde b}{p}) + n(\tilde b - \frac{\tilde b}{p}).
\]
By the hypothesis on $b$ and $b'$ in (\ref{E: b restrictions}), it follows that $C(l, n) \in L(b'/q; \rho_{l, n})$, where
\[
\rho_{l, n} = (l+n) \ep + n(\tilde b - b) + m( \frac{1}{p-1} - \frac{\tilde b}{p} ).
\]
\end{proof}

We may now define the Frobenius operator $\beta_\kappa$. First, define the (geometric Frobenius) operator $\psi_t: S(b'/p, \ep) \rightarrow S(b', \ep)$ by 
\[
\psi_t: \quad \sum_{n, m \geq 0} A(n, m)  t^n w^{(m)} \longmapsto \sum_{n, m \geq 0} A(p n, m) t^n w^{(m)},
\]
and set
\[
\beta_\kappa := \psi_t^a \circ [\bar \alpha_a]_\kappa: S(b', \ep) \rightarrow S(b', \ep).
\]
This is a completely continuous operator, and so its Fredholm determinant is $p$-adic entire. The Dwork trace formula now proves the following (see Section \ref{SS: dwork trace proof} for the proof):

\begin{theorem}\label{T: dwork trace}
$L(\text{Sym}^{\infty, \kappa} Kl, T) = det(1 - \beta_\kappa T \mid S(b', \ep))^{\delta_q}$ where $\delta_q$ sends any function $g(T)$ to $g(T)/g(qT)$.
\end{theorem}

We now rewrite this using the cohomology theory introduced in Section \ref{S: symmetric power cohom}. First, we note that:

\begin{lemma}\label{L: partial beta}
$q \partial_\kappa \circ \beta_\kappa = \beta_\kappa \circ \partial_\kappa$.
\end{lemma}

This lemma is a bit involved and so we delay its proof until Section \ref{SS: proof of commute of partial}. As a consequence of Lemma \ref{L: partial beta}, $\beta_\kappa$ induces maps $\bar \beta_\kappa: H^1_\kappa(S(b', \ep)) \rightarrow H^1_\kappa(S(b', \ep))$ and $\bar \beta_\kappa: H^0_\kappa(S(b', \ep)) \rightarrow H^0_\kappa(S(b', \ep))$ with the property:
\[
L(\text{Sym}^{\infty, \kappa} Kl, T) = \frac{det(1 - \bar \beta_\kappa T \mid H^1_\kappa(S(b', \ep)))}{det(1 - q \bar \beta_\kappa T \mid H^0_\kappa(S(b', \ep)))}.
\]
Using our knowledge of cohomology from the previous sections gives:

\begin{theorem}
If $\kappa \in \bb Z_p$ then $L(\text{Sym}^{\infty, \kappa} Kl, T)$ is an entire function. When $\kappa = 0$, then
\[
L(\text{Sym}^{\infty, 0} Kl, T) = \frac{det(1 - \bar \beta_0 T \mid H^1_\kappa(S(b', \ep)))}{1 - qT},
\]
which is still an entire function by Theorem \ref{T: T=1 root} below. 
\end{theorem}

\begin{proof}
If $\kappa \not= 0$, then the first statement follows immediately from Theorem \ref{T: 0-homology}. Suppose now that $\kappa = 0$. In this case, the constant 1 is a basis for $H^0_\kappa$ with trivial action of Frobenius:
\[
q \bar \beta_0(1) = q \psi_t^a \circ [\bar \alpha_a]_0(1) = q \psi_t^a ( \Upsilon_q \circ \bar \alpha_a(1) )^0 = q.
\]
This proves the first part of the theorem. To show entireness, since the unit root $L$-function with $\kappa = 0$ takes the form $L_\text{unit}(0, T) = (1-T)/(1-qT)$, we see that
\[
L(\text{Sym}^{\infty, 0} Kl, T) = (1-T) \frac{L(\text{Sym}^{\infty, -2} Kl, qT)}{1-qT}.
\]
We will see in the next proposition that $L(\text{Sym}^{\infty, -2} Kl, T)$ has a root at $T = 1$, which proves the entireness for $\kappa = 0$.
\end{proof}

\begin{theorem}\label{T: T=1 root}
For every $\kappa \in \bb Z_p$, $T=1$ is a root of $L(\text{Sym}^{\infty, \kappa} Kl, T)$.
\end{theorem}

\begin{proof}
While one may use the cohomology above to prove this, we will use a different argument. From \cite[Theorem B]{Robba-SymmetricPowersof-1986}  $T=1$ is a root of the $k$-th symmetric power $L$-function $L(Sym^k Kl, T)$ for every positive integer $k$. The result now follows by continuity: let $\{ k_m \}$ be any sequence of positive integers which tend to infinity and $k_m \rightarrow \kappa$ $p$-adically, then
\[
\lim_{m \rightarrow \infty} L(Sym^{k_m} Kl, T) = L(Sym^{\infty, \kappa} Kl, T).
\]
\end{proof}

Now that $L(\text{Sym}^{\infty, \kappa} Kl, T)$ is an entire function, we next investigate the zeros using the $q$-adic Newton polygon. We could use the above theory to give a lower bound for the Newton polygon, however, since we are using the ``first'' splitting function, the estimate is weaker than it could be. In Section \ref{S: estimates} we switch to using the infinite splitting function which will provide a stronger estimate. This switch comes at the cost of reworking much of the theory we just established because the map $\beta_\kappa$ and the spaces involved are altered.

\section{$p$-adic estimates}\label{S: estimates}

In order to obtain the best possible $p$-adic estimates, we modify the splitting function used earlier. Unfortunately, this affects nearly all the spaces and maps defined earlier as well, which means we need to redefine them accordingly, as well as reprove certain results. We do this now.\footnote{The paper could have been written using only the results from this section, however, I feel there is something to be gained from the simplicity using the first splitting function as well as potential future work.}

Let $\gamma \in \overline{\bb Q}_p$ be a root of $\sum_{i=0}^\infty \frac{t^{p^i}}{p^i}$ with $ord_p(\gamma) = \frac{1}{p-1}$, and set $\Omega_0 := \bb Q_p(\gamma)$. Let $E(t) := \exp\left( \sum_{i=0}^\infty \frac{t^{p^i}}{p^i} \right)$ be the Artin-Hasse exponential. Define $\gamma_l := \sum_{i=0}^l \frac{\gamma^{p^i}}{p^i}$ and note that $ord_p(\gamma_l) \geq \frac{p^{l+1}}{p-1}-l-1$. Dwork's infinite splitting function is defined as $\theta_\infty(t) := E(\gamma t) = \sum_{i=0}^\infty \lambda_i t^i$, and it is well-known that the coefficients satisfies $ord_p(\lambda_i) \geq \frac{i}{p-1}$ since the Artin-Hasse exponential has $p$-adic integral coefficients. It also satisfies the same properties as $\theta(z)$ defined in Section \ref{S: rel bes} in terms of being a $p$-adic analytic lift of an additive character on $\bb F_q$. Set $\hat b := p / (p-1)$. Throughout this section, let $b$ and $b'$ be real numbers satisfying:
\begin{equation}\label{E: b's - 2}
\hat b \geq b > 1/(p-1) \qquad \text{and} \qquad  b \geq b', \qquad \text{and set } \qquad \ep := b - \frac{1}{p-1}.
\end{equation}
With $q = p^a$ for $a \geq 0$, define the spaces
\begin{align*}
\c L(b'; \rho) &:= \left\{ \sum_{n=0}^\infty A(n) t^n \mid A(n) \in \Omega_0, ord_p A(n) \geq 2b' n + \rho \right\} \\
\c L(b') &:= \bigcup_{\rho \in \bb R} \c L(b; \rho) \\
\c K_q(b', b; \rho) &:= \left\{ \sum_{n \geq 0, u \in \bb Z} A(n, u) t^n \cdot t^{q m(u)} x^u \mid A(n, u) \in \Omega_0, ord_p A(n, u) \geq 2b'n + b|u| + \rho  \right\} \\
\c K_q(b',b) &:= \bigcup_{\rho \in \bb R} \c K(b',b;\rho).
\end{align*}
We will denote $\c K_1(b', b)$ by $\c K(b', b)$. Note that these spaces are precisely the same as in Section \ref{S: rel bes} but with $\pi$ replaced by $\gamma$.

\medskip\noindent{\bf Relative cohomology.} Our first step is to obtain relative cohomology, which requires a boundary map $\widehat D_t$ defined as follows. As we did before, consider the $p$-adic analogue of $\theta(1)^{Tr_{\bb F_q / \bb F_p}( x + \frac{t}{x})} = \theta(1)^{Tr(x)} \theta(1)^{Tr(t/x)}$: define
\begin{align*}
F_\infty(t,x) &:= \theta_\infty(x) \theta_\infty(\frac{t}{x}) \\
F_{\infty, a}(t,x) &:= \prod_{i=0}^{a-1} F_\infty(t^{p^i}, x^{p^i}),
\end{align*}
Next define a function $G(t,x)$ such that $F_\infty(t,x) = \frac{G(t,x)}{G(t^p, x^p)}$. Using this equation recursively, we see that $G(t,x)$ must be defined by
\[
G(t,x) := \prod_{j=0}^\infty F_\infty(t^{p^j}, x^{p^j}) \in \Omega_0[[t, x]].
\]
Define the (twisted) boundary operator $\widehat D_t$ on $\c K(b',b)$ by
\begin{align*}
\widehat D_t :&= \frac{1}{G(t,x)} \circ x \frac{\partial}{\partial x} \circ G(t,x) \\
&= x\frac{\partial}{\partial x} + W_1(t,x)
\end{align*}
where, setting $f_x := x \frac{\partial}{\partial x} f(t,x)$,
\begin{equation}\label{E: W1}
W_1(t,x) := \sum_{j=0}^\infty \gamma_j p^j f_x(t^{p^j}, x^{p^j}).
\end{equation}
As $W_1 \in \c K(\hat b, \hat b; -1)$ and acts via multiplication, $\widehat D_t$ is a well-defined endomorphism of $\c K(b', b)$. This allows us to define the relative cohomology spaces
\[
H_t^0(\c K(b', b)) := ker( \widehat D_t \mid \c K(b',b) ) \qquad \text{and} \qquad H_t^1(\c K(b', b)) := \c K(b',b) / \widehat D_t \c K(b',b).
\]
We will now identify $H_t^1(\c K(b', b))$ with a simpler space and show $H_t^0 = 0$. Define the spaces (with $\rho \in \bb R$)
\begin{align*}
\c V_q(b', b; \rho) &:= \left( \Omega_0[[t]] + \Omega_0[[t]] \frac{t^q}{x} \right) \cap \c K_q(b', b; \rho) \\
\c V_q(b', b) &:= \bigcup_{\rho \in \bb R} \c V_q(b', b; \rho).
\end{align*}

\begin{theorem}\label{T: rel cohom 2}
We have 
\[
H_{t^q}^0(\c K(b', b)) = 0 \qquad \text{and} \qquad H_{t^q}^1(\c K(b'/q, b)) \cong \c V_q(b'/q, b).
\]
Furthermore,
\[
\c K_q(b'/q, b; 0) = \c V_q(b'/q, b; 0) \oplus \widehat D_{t^q} \c K_q(b'/q, b; \ep).
\]
\end{theorem}

\begin{proof}
The result follows from \cite[Lemma 2.1]{MR0387281}, and a similar argument to \cite[Section 3.3]{MR2873140}.
\end{proof}

\medskip\noindent{\bf Relative Frobenius.} Define the relative Frobenius 
\[
\alpha_{\infty, 1}(t) := \psi_x \circ F_\infty(t, x),
\]
and, for $m \geq 1$,
\begin{align*}
\alpha_{\infty, m}(t) :&= \psi^m_x \circ F_{\infty, m}(t, x) \\
&= \alpha_{\infty, 1}(t^{p^{m-1}}) \circ \cdots \circ \alpha_{\infty, 1}(t^p) \circ \alpha_{\infty, 1}(t).
\end{align*}
It follows from $F_\infty(t, x) \in \c K(\hat b / p, \hat b / p; 0)$ that $F_\infty(t^{p^i}, x^{p^i}) \in \c K(\hat b/ p^{i+1}, \hat b / p^{i+1}; 0 )$, and thus $\alpha_{\infty, m}(t): \c K(b', b; 0) \rightarrow \c K_{p^m}(b' / p^m, b; 0)$. As
\[
p^m \widehat D_{t^{p^m}} \circ \alpha_{\infty, m} = \alpha_{\infty, m} \circ \widehat D_t,
\]
we see that $\alpha_{\infty, m}$ induces a map on relative cohomology $\bar \alpha_{\infty, m}(t): H_t^1(\c K(b', b)) \rightarrow H_{t^{p^m}}^1(\c K(b'/p^m, b))$. Again, we need to understand the power series entries of the matrix of $\bar \alpha_{\infty, m}(t)$:

\begin{theorem}\label{T: Dwork rel cohom 2}
With $\{1, \gamma t/x\}$ and $\{1, \gamma t^{p^m} / x\}$ as bases of $H_t^1(\c K(\hat b, \hat b))$ and $H_{t^{p^m}}^1(\c K(\hat b/p^m, \hat b))$, respectively, the matrix of the relative Frobenius $\bar \alpha_{\infty, m}$ satisfies
\[
\bar \alpha_{\infty, m}(1) = A_{m, 1}(t) + A_{m, 3}(t) \frac{\gamma t^{p^m}}{x} \qquad \text{and} \qquad \bar \alpha_{\infty, m}(\frac{\pi t}{x}) = A_{m, 2}(t) + A_{m, 4}(t) \frac{\gamma t^{p^m}}{x},
\]
where
\begin{equation}
\begin{split}
A_{m, 1} &\in \c L(\hat b/p^m ; 0) \\
A_{m, 2} &\in \c L(\hat b / p^m; \frac{1}{p-1} - \frac{\hat b}{p})
\end{split}
\qquad  \qquad 
\begin{split}
A_{m, 3} &\in \c L(\hat b/p^m; \hat b - \frac{1}{p-1}) \\
A_{m, 4} &\in \c L(\hat b/p^m; \hat b - \frac{\hat b}{p})
\end{split}
\end{equation}
and $A_{m, 1}(0) = 1$.
\end{theorem}

\begin{proof}
This follows from an analogous argument to \cite[Section 3]{AdolpSperb-ExponentialSumsand-1989} or \cite[Section 3]{MR2873140}.
\end{proof}

We now move on to the infinite symmetric power theory.

\subsection{$Sym^{\infty, \kappa}$-cohomology (again)}\label{S: symmetric power cohom again}

Analogously to what we did in Section \ref{S: symmetric power cohom}, within $\Omega_0[[t, w]]$ and setting $\w{m} := \kap{m} w^m$, define the spaces
\begin{align*}
\c S(b', \ep; \rho) &:= \left\{ \sum_{n, m \geq 0} A(n, m) t^n  \w{m} \mid A(n, m) \in \Omega_0, ord_p \> A(n, m) \geq 2b' n + \ep m + \rho \right\} \\
\c S(b', \ep) &:= \bigcup_{\rho \in \bb R} \c S(b', b; \rho).
\end{align*}
We now use the Gauss-Manin connection to define a boundary operator for this space. As our cohomology has changed since Section \ref{S: symmetric power cohom}, so has the connection. Define
\begin{align*}
\widehat \partial :&= \frac{1}{G(t,x)} \circ t \frac{\partial}{\partial t} \circ G(t,x) \\
&= t\frac{\partial}{\partial t} + W_2(t,x)
\end{align*}
where, setting $f_t(t,x) := t \frac{\partial}{\partial t} f(t,x)$, 
\[
W_2(t,x) := \sum_{j=0}^\infty \gamma_j p^j f_t(t^{p^j}, x^{p^j}) = \sum_{j=0}^\infty \gamma_j p^j \left(\frac{t}{x}\right)^{p^j}.
\]
Note that $\widehat \partial$ is an endomorphism of $\c K(b', b)$ since $W_2(t,x) \in \c K(\hat b, \hat b; -1)$. Also, $\widehat \partial$ commutes with $\widehat D(t)$ as endomorphisms of $\c K(b',b)$, and thus it induces an operator on relative cohomology $\widehat \partial: H_t^1(\c K(b', b)) \rightarrow H_t^1(\c K(b', b))$. With respect to the basis $\{1, \frac{\gamma t}{x} \}$ on $H_t^1(\c K(b', b))$, we may write the boundary map in matrix form $\widehat \partial = t \frac{d}{dt} + \widehat H$, where $\widehat H$ is a two-by-two matrix with entries in $\c L(b')$. We define the $\kappa$-symmetric power of the boundary operator $\widehat \partial$ on $\c S(b', \ep)$ by $\widehat \partial_\kappa := t \frac{d}{dt} + \c L_{\kappa, \widehat H}$, where $\c L_{\kappa, \widehat H}$ is defined similarly to (\ref{E: Lie derivative}) but using the matrix $\widehat H$. Define the cohomology spaces
\[
H_\kappa^0(\c S(b', \ep)) := ker( \widehat \partial_\kappa \mid \c S(b', \ep)) \qquad \text{and} \qquad H_\kappa^1(\c S(b', \ep)) := \c S(b', \ep) / \widehat \partial_\kappa \c S(b', \ep).
\]

\subsection{$H^0_\kappa$ and $H_\kappa^1$ when $\kappa \in \bb Z_p \setminus \bb Z_{\geq 0}$}

Throughout this section we will assume $\kappa \in \bb Z_p \setminus \bb Z_{\geq 0}$. As we did before, we will show $H^0_\kappa = 0$ and identify $H^1_\kappa$ with a simpler space $\c R(b')$ defined as follows. The spaces $\c L(b'; \rho)$ and $\c L(b')$ sit naturally in $\c S(b', \ep)$ by sending $\xi \mapsto \xi$, the series with no $\w{m}$ terms. We will denote by $\c R(b')$ and $\c R(b'; \rho)$ the images of these spaces in $\c S(b', \ep)$. We will show that $H_{\kappa}^1(S(b', \ep))  \cong R(b')$.

Abusing notation slightly, define the matrix
\[
H := 
\left(
\begin{array}{cc}
0  & 1     \\
\gamma^2 t  &  0     
\end{array}
\right).
\]
(Note, this is the same matrix as in Section \ref{S: symmetric power cohom} with $\pi$ replaced by $\gamma$.) In the following lemma we relate $\widehat H$ with the matrix $H$. This lemma is a technical key which allows us to prove results using the simpler $H$ matrix and then lift the result to that of the more complicated $\widehat H$.

\begin{lemma}\label{L: technical}
There exists $\eta_0$ a 1-unit in $\c L(b'; 0)$ and a matrix $\widehat R$ such that
\[
\widehat H =  \eta_0 H + \widehat R,
\]
where
\[
\widehat R = 
\left(
\begin{array}{cc}
 \widehat R_{00} & \widehat R_{01}   \\
 \widehat R_{10} & \widehat R_{11} 
\end{array}
\right)
\]
satisfies
\[
\widehat R_{00} \in \c L(b'; 0) \quad \widehat R_{01} \in \c L(b'; \ep) \quad \widehat R_{10} \in \c L(b'; -\ep) \quad \widehat R_{11} \in \c L(b'; 0).
\]
\end{lemma}

\begin{proof}
Write $f_x(t^{p^j}, x^{p^j}) = f_x(t, x)^{p^j} + p h(t, x)$, where $h$ has $p$-adic integral coefficients. Using this, define $Q_1$ and $R_1$ such that
\[
W_1 = \gamma f_x Q_1 + R_1.
\]
It follows from (\ref{E: W1}) that $R_1, Q_1 \in K(\hat b, \hat b; 0)$ and $Q_1$ is a 1-unit. Similarly, define $Q_2$ (no $R_2$ is required) such that $W_2 = \gamma f_t Q_2$, and note that $Q_2 \in \c K(\hat b, \hat b; 0)$ is a 1-unit. 

Consider now $\widehat \partial(1) = W_2$. From \cite[Lemma 2.1]{MR0387281}, there exists a 1-unit $\eta_0 \in \c L(b'; 0)$ and $h \in \c K(b', b; \ep)$ such that $Q_2 = \eta_0 + \gamma f_x h$. Then
\begin{align*}
W_2 &= \gamma f_t Q_2 \\
&= \gamma f_t \eta_0 + \gamma f_x( \gamma f_t h) \\
&= \gamma f_t \eta_0 + (W_1 - R_1) Q_1^{-1}(\gamma f_t h) \\
&= \gamma f_t \eta_0 + \widehat D_t( \zeta_1) - \xi_1
\end{align*}
where $\zeta_1 \in \c K(b', b; 0)$ and $\xi_1 \in \c K(b', b; 0)$. By Theorem \ref{T: rel cohom 2}, write $\xi_1 = \xi_1^{(0)} + \xi_1^{(1)} \frac{\gamma t}{x} + \widehat D_t(\zeta_2)$ where $\xi_1^{(0)} \in \c L(b'; 0)$, $\xi_1^{(1)} \in \c L(b'; \ep)$, and $\zeta_2 \in \c K(b', b; \ep)$. Then
\[
\widehat \partial(1) = \eta_0 \frac{\gamma t}{x} + \xi_1^{(0)} + \xi_1^{(1)} \frac{\gamma t}{x} + \widehat D_t(\zeta_1 + \zeta_2).
\]
This shows, $\widehat R_{00} = \xi_1^{(0)} \in \c L(b'; 0)$ and $\widehat R_{01} = \xi_1^{(1)} \in \c L(b'; \ep)$.

We now compute $\widehat \partial(\frac{\gamma t}{x})$. Write
\begin{align*}
\widehat \partial \left(\frac{\gamma t}{x} \right) &= \frac{\gamma t}{x} + W_2 \frac{\gamma t}{x} \\
&= \frac{\gamma t}{x} + \gamma f_t Q_2 \frac{\gamma t}{x}  \\
&= \frac{\gamma t}{x}  + \gamma f_t( \eta_0 + \gamma f_x h) \frac{\gamma t}{x} \\
&= \frac{\gamma t}{x}  + \eta_0 \left( \frac{\gamma t}{x} \right)^2 + \gamma f_x( \gamma f_t h) \frac{\gamma t}{x}.
\end{align*}
Now,
\begin{align*}
\widehat D_t \left(\frac{\gamma t}{x} \right) &= -\frac{\gamma t}{x}  + W_1 \frac{\gamma t}{x}  \\
&= - \frac{\gamma t}{x}  + (\gamma f_x Q_1 + R_1) \frac{\gamma t}{x}  \\
&= - \frac{\gamma t}{x}  + \gamma^2t - \left(\frac{\gamma t}{x} \right)^2 + \gamma f_x \xi_1 + \xi_2
\end{align*}
where $\xi_1 := (Q_1 - 1) \frac{\gamma t}{x} $ and $\xi_2 := R_1 \frac{\gamma t}{x} $ are elements in $\c K(b', b; -\ep)$. By Theorem \ref{T: rel cohom 2},
\[
\gamma f_x \xi_1 + \xi_2 = \tilde \eta_0 + \tilde \eta_1 \frac{\gamma t}{x}  + \widehat D_t(\tilde \zeta)
\]
for some $\tilde \eta_0 \in \c L(b'; -\ep)$, $\tilde \eta_1 \in \c L(b'; 0)$, and $\tilde \zeta \in \c K(b', b; 0)$. Next, setting $\tilde \xi := \gamma f_t h \frac{\gamma t}{x}  \in \c K(b', b; -\ep)$, then
\[
\gamma f_x( \pi f_t h \frac{\gamma t}{x} ) = \gamma f_x \tilde \xi = \widehat D_t( \zeta_1) + \zeta_2 + \zeta_3 \frac{\gamma t}{x} 
\]
where $\zeta_1 \in \c K(b', b; -\ep)$, $\zeta_2 \in \c L(b'; -\ep)$, and $\zeta_3 \in \c L(b'; 0)$. Consequently, 
\begin{align*}
\widehat \partial (\frac{\gamma t}{x} ) &= \frac{\gamma t}{x}  + \eta_0( \frac{\gamma t}{x} )^2 + \gamma f_x( \gamma f_t h \frac{\gamma t}{x} ) \\
&= \frac{\gamma t}{x}  + \eta_0 \left( - \frac{\gamma t}{x} + \pi^2 t + \tilde \eta_0 + \tilde \eta_1 \frac{\gamma t}{x}  + \widehat D_t(\tilde \zeta) \right) + \widehat D_t(\zeta_1) + \zeta_2 + \zeta_3 (\frac{\gamma t}{x} ) \\
&= \eta_0 \gamma^2 t + \widehat R_{10} + \widehat R_{11} \frac{\gamma t}{x}  + \widehat D_t(\zeta_1 + \eta_0 \tilde \zeta),
\end{align*}
where
\begin{align*}
\widehat R_{10} &:= \eta_0 \tilde \eta_0 + \zeta_2 \in \c L(b'; -\ep) \\
\widehat R_{11} &:= 1 - \eta_0 + \eta_0 \tilde \eta_1 + \zeta_3 \in \c L(b'; 0).
\end{align*}
This finished the proof.
\end{proof}

It follows immediately from the estimates on the matrix of $\widehat R$ that:

\begin{corollary}\label{C: technical-2}
We have
\begin{enumerate}
\item $\c L_{\kappa, \widehat H} = \eta_0 \c L_{\kappa, H} + \c L_{\kappa, \widehat R}$, where $\c L_{\kappa, H}$ and $\c L_{\kappa, \widehat R}$ is defined similarly to (\ref{E: Lie derivative}).
\item $\c L_{\kappa, \widehat R} \c S(b', \ep; 0) \subset \c S(b', \ep; 0)$.
\end{enumerate}
\end{corollary}

\begin{lemma}\label{L: dwork decomp for hat H}
$\c S(b', \ep; 0) = \c R(b', \ep; 0) + \c L_{\kappa, \widehat{H}} \c S(b', \ep; \ep)$.
\end{lemma}

\begin{proof}
Let $\xi \in \c S(b', \ep; 0)$. We will show $\xi$ is contained in the righthand side. By the proof of Lemma \ref{L: Dwork decomp H} but with $\pi$ replaced by $\gamma$ in the matrix of $H$, there exists $\eta \in \c R(b', \ep; 0)$ and $\zeta \in \c S(b', \ep; \ep)$ such that $\xi  = \eta + \c L_{\kappa, H}(\zeta)$. By Corollary \ref{C: technical-2}, we may write $\c L_{\kappa, H} = (\c L_{\kappa, \widehat{H}} - \c L_{\kappa, \widehat R}) \eta_0^{-1}$, and so
\[
\xi = \eta + (\c L_{\kappa, \widehat{H}} - \c L_{\kappa, \widehat R})\eta_0^{-1} \zeta = \eta + \c L_{\kappa, \widehat{H}}(\eta_0^{-1} \zeta) - \c L_{\kappa, \widehat R}(\eta_0^{-1} \zeta).
\]
By Corollary \ref{C: technical-2}, since $\eta_0^{-1} \zeta \in \c S(b', \ep; \ep)$, $\c L_{\kappa, \widehat R}(\eta_0^{-1} \zeta) \in \c S(b', \ep; \ep)$. We may now repeat this procedure with $\c L_{\kappa, \widehat R}(\eta_0^{-1} \zeta)$ and so forth, thus showing $\xi \in \c R(b', \ep; 0) + \c L_{\kappa, \widehat{H}} \c S(b', \ep; \ep)$.

To prove the other direction, let $\zeta \in \c S(b', \ep; \ep)$. Again by Corollary \ref{C: technical-2}, $\c L_{\kappa, \widehat{H}}(\zeta) = \eta_0 \c L_{\kappa, H}(\zeta) + \c L_{\kappa, \widehat R}(\zeta)$. Now $\c L_{\kappa, H}(\zeta) \in \c S(b', \ep; 0)$, and Corollary \ref{C: technical-2} gives $ \c L_{\kappa, \widehat R}(\zeta) \in \c S(b', \ep; \ep)$. This proves the result.
\end{proof}

\begin{lemma}\label{L: V and L_H complement}
$\c R(b', \ep) \cap \c L_{\kappa, H} \c S(b', \ep) = \{ 0\}$. 
\end{lemma}

\begin{proof}
Let  $\xi = \sum_{m \geq 0} \xi_m \w{m} \in \c S(b', \ep)$ and $\eta \in \c R(b', \ep)$ be such that $\c L_{\kappa, H} \xi = \eta$. This is equivalent to the system of equations
\[
\begin{cases}
\eta &= \quad \qquad \quad + \> \kappa \gamma^2 t \xi_1 \\
0 &=  \xi_0 \quad \qquad + 2 (\kappa-1)\gamma^2 t \xi_2 \\
0 &= \xi_1 \quad \qquad + 3 (\kappa-2) \gamma^2 t \xi_3 \\
0 &= \xi_2 \quad \qquad + 4 (\kappa-3) \gamma^2 t \xi_4 \\
&\vdots
\end{cases}
\]
Observe that the second equation shows $t \mid \xi_0$, and the fourth equation shows $t \mid \xi_2$, and hence $t^2 \mid \xi_0$. Repeating this shows $t^m \mid \xi_0$ for any $m \geq 1$. Hence, $\xi_0 = 0$, and thus $\xi_{2m} = 0$ for $m \geq 1$. A similar argument using the odd rows shows $\xi_{2m+1} = 0$ for $m \geq 0$. 
\end{proof}

\begin{theorem}\label{T: cohom on hat partial}
Let $\kappa \in \bb Z_p \setminus \bb Z_{\geq 0}$. Then $H^0_\kappa(\c S(b', \ep)) = 0$, and $H^1_\kappa(\c S(b', \ep)) \cong \c R(b')$. Furthermore, 
\begin{equation}\label{E: dwork decomp sym inf}
\c S(b', \ep; 0) = \c R(b', \ep; 0) \oplus \widehat \partial_\kappa \c S(b', \ep; \ep).
\end{equation}
\end{theorem}

\begin{proof}
First note that the righthand side of (\ref{E: dwork decomp sym inf}) is contained in the left. Next, let $\xi \in \c S(b', \ep; 0)$ and set $E := t \frac{d}{dt}$.  By Lemma \ref{L: dwork decomp for hat H}, there exists $\eta \in \c R(b', \ep; 0)$ and $\zeta \in \c S(b', \ep; \ep)$ such that
\begin{align*}
\xi &= \eta + \c L_{\kappa, \widehat{H}}(\zeta) \\
&= \eta + \widehat \partial_\kappa(\zeta) - E(\zeta).
\end{align*}
As $E(\zeta) \in \c S(b', \ep; \ep)$, we may repeat this procedure to obtain $\c S(b', \ep; 0) = \c R(b', \ep; 0) + \widehat \partial_\kappa \c S(b', \ep; \ep)$.

We now show directness. Let $\eta \in \c R(b', \ep) \cap \widehat \partial_\kappa \c S(b', \ep)$. Let $\zeta \in \c S(b', \ep)$ be such that $\widehat \partial_\kappa \zeta = \eta$. If $\zeta \not= 0$, then there exists $c \in \bb R$ such that $\zeta \in \c S(b', \ep; c)$ but $\zeta \not\in \c S(b', \ep; c + \ep)$. Now, by Corollary \ref{C: technical-2},
\[
\eta = \widehat \partial_\kappa \zeta = E \zeta + \eta_0 \c L_{\kappa, H} \zeta + \c L_{\kappa, \widehat R} \zeta.
\]
Set $\xi_1 := E \zeta + \c L_{\kappa, \widehat R} \zeta \in \c S(b', \ep; c)$, By the first part of this proof, there exists $\eta_1 \in \c R(b', \ep; c)$ and $\zeta_1 \in \c S(b', \ep; c + \ep)$ such that
\[
\xi_1 = \eta_1 + \widehat \partial_\kappa \zeta_1 = \eta_1 + E_1 \zeta_1 + \eta_0 \c L_{\kappa, H}(\zeta_1) + \c L_{\kappa, \widehat R}(\zeta_1).
\]
Set $\xi_2 := E(\zeta_1) + \c L_{\kappa, \widehat R} \zeta_1 \in \c S(b', \ep; c + \ep)$. Iterating this procedure, we obtain
\[
\eta = \sum_{i=1}^\infty \eta_i + \eta_0 \c L_{\kappa, H} (\zeta + \sum_{i=1}^\infty \zeta_i),
\]
which we rewrite as
\[
\c L_{\kappa, H}(\zeta + \sum_{i \geq 1} \zeta_i) = (\eta - \sum_{i \geq 1} \eta_i) \eta_0^{-1}.
\]
By Lemma \ref{L: V and L_H complement}, we must have $\zeta = - \sum_{i \geq 1} \zeta_i \in \c S(b', \ep; c + \ep)$, which contradicts our choice of $c$. 

Lastly, observe that setting $\eta = 0$ shows $\text{ker } \widehat \partial_\kappa = 0$.
\end{proof}

\subsection{Frobenius and estimates}\label{SS: Frob-2}

Now that we have finished the study of cohomology, we move on to the Frobenius. Most of the arguments are the same as in Section \ref{SS: Frob}. Define the $\kappa$-symmetric power of $\bar \alpha_{\infty, a}$ as follows: define $[\bar \alpha_{\infty, a}]_\kappa: \c S(b', \ep) \rightarrow \Omega_0[[t, w]]$ by linearly extending over $\c L(b')$ the action
\[
[\bar \alpha_{\infty, a}]_\kappa( \w{m} ) := \kappa^{\u{m}} \cdot ( \Upsilon_q \circ \bar \alpha_{\infty, a}(1) )^{\kappa - m} (\Upsilon_q \circ \bar \alpha_{\infty, a} \frac{\gamma t}{x} )^m,
\]
where $\Upsilon_q: H_{t^q}^1(\c K_q(b'/q, b)) \rightarrow \Omega_0[[t, w]]$  by sending $\zeta + \xi \frac{\gamma t^q}{x} \longmapsto \zeta + \xi w$. Just as in Lemma \ref{L: well-defined}, $[\bar \alpha_{\infty, a}]_\kappa$ is an endomorphism of $\c S(b', \ep)$ when $\hat b - \frac{1}{p-1} \geq b > \frac{1}{p-1}$ and $b \geq b'$. (Note that $p \geq 5$ ensures that $\hat b > 2/(p-1)$ so that such a $b$ exists.) Define the Frobenius map
\[
\beta_{\infty, \kappa} := \psi_t^a \circ [\alpha_{\infty, a}]_\kappa: \c S(b', \ep) \rightarrow \c S(b', \ep).
\]
An analogous result to Lemma \ref{L: partial beta} shows $q \widehat \partial_\kappa \circ \beta_{\infty, \kappa} = \beta_{\infty, \kappa} \circ \widehat \partial_\kappa$, and so $\beta_{\infty, \kappa}$ induces maps on cohomology $\bar \beta_{\infty, \kappa}: H^0_\kappa(\c S(b', \ep)) \rightarrow H^0_\kappa(\c S(b', \ep))$ and $\bar \beta_{\infty, \kappa}: H^1_\kappa(\c S(b', \ep)) \rightarrow H^1_\kappa(\c S(b', \ep))$. Combining this with the Dwork trace formula (analogous to Theorem \ref{T: dwork trace}), and Theorem \ref{T: cohom on hat partial} gives:

\begin{theorem}
Set $\hat b - \frac{1}{p-1} \geq b > \frac{1}{p-1}$ and $b \geq b'$. For $\kappa \in \bb Z_p \setminus \bb Z_{\geq 0}$,
\[
L(Sym^{\infty, \kappa} Kl, T) = det(1 - \bar \beta_{\infty, \kappa} T \mid H^1_\kappa(\c S(b', \ep))).
\]
\end{theorem}

We are now able to finish this section by providing an estimate for the $q$-adic Newton polygon of $L(Sym^{\infty, \kappa} Kl, T)$ for every $\kappa$. While we have concentrated in this section on the case when $\kappa$ is not a positive integer, the function $L(Sym^{\infty, \kappa} Kl, T)$ is continuous in the variable $\kappa$. Thus, we need only prove the result for $\kappa \in \bb Z_p \setminus \bb Z_{\geq 0}$ for the estimate to hold.

\begin{theorem}\label{T: NP estimate}
Let $\kappa \in \bb Z_p$. Writing $L(Sym^{\infty, \kappa} Kl, T) = \sum_{m=0}^\infty c_m T^m$, then for every $m \geq 0$, 
\[
ord_q c_m \geq \left(1 - \frac{1}{p-1}\right) m(m-1).
\]
\end{theorem}

\begin{proof}
We will prove this assuming $\kappa \in \bb Z_p \setminus \bb Z_{\geq 0}$. As this set is dense in $\bb Z_p$, the result will follow by continuity in $\kappa$ of $L(Sym^{\infty, \kappa} Kl, T)$.

Defining the operator $\beta_{\infty, \kappa, 1} := \psi_t \circ [\bar \alpha_{\infty, 1}]_\kappa: \c S(b', \ep) \rightarrow \c S(b', \ep)$, then we see that
\begin{align*}
\beta_{\infty, \kappa, 1}^a &= \psi_t \circ [\bar \alpha_{\infty, 1}(t)]_\kappa \circ \cdots \circ \psi_t \circ [\bar \alpha_{\infty, 1}(t)]_\kappa \\
&= \psi_t^a \circ [ \bar \alpha_{\infty, 1}(t^{p^{a-1}})]_\kappa  \circ \cdots \circ [\bar \alpha_{\infty, 1}(t^p)]_\kappa \circ [\bar \alpha_{\infty, 1}(t) ]_\kappa \\
&= \psi_t^a \circ [ \bar \alpha_{\infty, 1}(t^{p^{a-1}}) \circ \cdots \circ \bar \alpha_{\infty, 1}(t^p) \circ \bar \alpha_{\infty, 1}(t) ]_\kappa \\
&= \psi_t^a \circ [ \bar \alpha_{\infty, a}(t) ]_\kappa \\
&= \beta_{\infty, \kappa, a},
\end{align*}
where we have used an argument similar to \cite[Corollary 2.4]{MR3249829} for the third equality. Hence, on cohomology, $\bar \beta_{\infty, \kappa, 1}^a = \bar \beta_{\infty, \kappa, a}$. 

Now,
\begin{align}\label{E: RelFrobRelation}
det(1 - \bar \beta_{\infty, \kappa, a} T^a \mid H^1_\kappa(\c S(b', \ep))) &= det(1 - \bar \beta_{\infty, \kappa, 1}^a T^a \mid H^1_\kappa(\c S(b', \ep))) \notag \\
&= \prod_{\zeta^a = 1} det(1 - \zeta \bar \beta_{\infty, \kappa, 1} T \mid H^1_\kappa(\c S(b', \ep))).
\end{align}
Counting multiplicities, let $m_i$ denote the number of reciprocal roots of $det(1 - \bar \beta_{\infty, \kappa, 1} T \mid H^1_\kappa(\c S(b', \ep)))$ which have slope $s_i$; note, we say $\lambda \in \bb C_p$ has slope $s_i$ if $ord_p(\lambda) = s_i$. Then, from (\ref{E: RelFrobRelation}),  $det(1 - \bar \beta_{\infty, \kappa, a} T \mid H^1_\kappa(\c S(b', \ep)))$ has $m_i$ reciprocal roots of slope $s_i / a$, or alternatively, it has $m_i$ reciprocal roots of $q$-adic slope $s_i$.

In order to have the map $\beta_{\infty, \kappa, 1}$ well-defined, we require $\hat b - \frac{1}{p-1} \geq b > \frac{1}{p-1}$. Thus, set $b$ to be the maximum value  $b = \hat b - \frac{1}{p-1}$. By definition, $\ep := b - \frac{1}{p-1} = \hat b - \frac{2}{p-1}$. Note that $\gamma^{2n} t^n \in \c R(\hat b/p; 0) \subset \c S(\hat b/p, \hat b - \frac{2}{p-1}; 0)$. As $[\bar \alpha_{\infty, 1}]_\kappa$ is well-defined on this space, we have $\beta_{\infty, \kappa, 1}(\gamma^{2n} t^n) \in \c S(\hat b, \ep; 0)$. Unfortunately, we are unable to control the reduction of this space in cohomology using Theorem \ref{T: cohom on hat partial} since we need ``$\ep = b - \frac{1}{p-1}$'' and ``$b \geq b'$''; in our case of $\c S(\hat b, \ep; 0)$, we have $b' = \hat b > b = \hat b - \frac{1}{p-1}$. We may fix this by viewing $\beta_{\infty, \kappa, 1}(\gamma^{2n} t^n) \in \c S(\hat b, \ep; 0) \subset \c S(\hat b - \frac{1}{p-1}, \ep; 0)$. Using Theorem \ref{T: cohom on hat partial}, $\bar \beta_{\infty, \kappa, 1}(\gamma^{2n} t^n) = \sum_{m \geq 0} B(m,n) t^m \in \c R(\hat b - \frac{1}{p-1}; 0)$. Writing $ \sum_{m \geq 0} B(m,n) t^m =  \sum_{m \geq 0} B(m,n)\gamma^{-2m} \cdot \gamma^{2m} t^m$, we see that $ord_p B(m, n) \gamma^{-2m} \geq 2(\hat b - 1/(p-1)) m - 2m/(p-1) = 2m(1 - \frac{1}{p-1})$. The result now follows from the previous paragraph and the argument given by Dwork in \cite[Section 7]{Dwork-zetafunctionof-1964}
\end{proof}

\section{Some delayed proofs}

\subsection{Proof of Theorem \ref{T: dwork trace}}\label{SS: dwork trace proof}.

In order to prove this result, we need to first recall the Dwork trace formula on the fibers. Fix $\bar t \in \overline{\bb F}_q^*$ and let $\hat t$ be its Teichm\"uller lift. Set $d(\bar t) := [\bb F_q(\bar t): \bb F_q]$ and define $q_{\bar t} := q^{d(\bar t)}$. Define
\[
\alpha_{\hat t} := \psi_x^{a d(\bar t)} \circ F_{a d(\bar t)}(\hat t, x).
\]
Observe that $\alpha_{\hat t}$ is an endomorphism of $K_{\hat t}(b)$, where $K_{\hat t}(b)$ denotes the space obtained from $K(b', b)$ by specializing $t = \hat t$. Dwork's trace formula states
\[
(q_{\bar t}^m-1)^n Tr(\alpha_{\hat t}^m \mid K_{\hat t}(b) ) = \sum_{\bar x \in \bb F_{q_{\bar t}^m}^*} \Psi \circ Tr_{\bb F_{q_{\bar t}^m} / \bb F_q}( x + \frac{\bar t}{x}),
\]
or equivalently
\[
L(Kl_{\bar t}, T) = \frac{det(1 - \alpha_{\hat t} T \mid K_{\hat t}(b))}{det(1 - q_{\bar t} \alpha_{\hat t} T \mid K_{\hat t}(b))}.
\]
By Theorem \ref{T: Dwork rel cohom}, the operator $D_{\hat t} := x \frac{d}{d x} + \pi \left( x - \frac{\hat t}{x} \right)$ acts on the space $K_{\hat t}(b)$ such that the associated cohomology satisfies $H^0( K_{\hat t}(b) ) := ker(D_{\hat t}) = 0$ and $H^1( K_{\hat t}(b)) := K_{\hat t}(b) / D_{\hat t} K_{\hat t}(b) \cong \Omega(\hat t) + \Omega(\hat t) \frac{\pi \hat t}{x}$. Further, the Frobenius $\alpha_{\hat t}$ induces a map $\bar \alpha_{\hat t}$ on cohomology satisfying
\begin{align*}
L(Kl_{\bar t}, T) &= det(1 - \bar \alpha_{\hat t} T \mid H^1(K_{\hat t}(b))) \\
&= (1 - \pi_0(\bar t) T)(1 - \pi_1(\bar t) T).
\end{align*}

We now move on to the infinite symmetric powers of the fibers, defined analogously to $S(b', \ep)$. Define $S_{\hat t}(\ep)$ as the space obtained from $S(b', \ep)$ by specializing $t = \hat t$. As a consequence of Theorem \ref{T: Dwork rel cohom}, observe that $\bar \alpha_{\hat t}(1) = 1 + \eta_{\hat t} + \zeta_{\hat t} \frac{1}{x}$ for some elements $\eta_{\hat t}, \zeta_{\hat t} \in \Omega(\hat t)$ satisfying $|\eta_{\hat t}|_p < 1$ and $|\zeta_{\hat t}|_p < 1$. Define $\Upsilon_{q_{\bar t}}: H^1(K_{\hat t}(b)) \rightarrow S_{\hat t}(\ep)$ by $\zeta + \xi \frac{\pi \hat t}{x} \mapsto \zeta + \xi w$. For any $\tau \in \bb Z_p$, $(\Upsilon_{q_{\bar t}} \circ \alpha_{\hat t}(1))^{\tau}$ is a well-defined element of $\Omega(\hat t)[[w]]$. Define $[\bar \alpha_{\hat t}]_\kappa$ acting on $S_{\hat t}(\ep)$ by linearly extending
\[
[\bar \alpha_{\hat t}]_\kappa( \w{m} ) := \kappa^{\u{m}} (\Upsilon_{q_{\bar t}} \circ \bar \alpha_{\hat t}(1) )^{\kappa - m} (\Upsilon_{q_{\bar t}} \circ \bar \alpha_{\hat t} \frac{\pi \hat t}{x})^m.
\]
By Lemma \ref{L: well-defined}, $[\bar \alpha_{\hat t}]_\kappa$ is a well-defined endomorphism of $S_{\hat t}(\ep)$ when $\tilde b - \frac{1}{p-1} \geq b$. The main purpose for working on the fibers is that, by an argument similar to \cite[Corollary 2.4, part 2]{MR3249829}, we have
\begin{equation}\label{E: local factor}
det(1 -[\bar \alpha_{\hat t}]_\kappa T \mid S_{\hat t}(\ep) ) = \prod_{m=0}^\infty \left(1 - \pi_0(\bar t)^{\kappa - m} \pi_1(\bar t)^m T \right),
\end{equation}
which is the local factor in the Euler product of $L(Sym^{\kappa, \infty} Kl, T)$. We may now prove the theorem.

\begin{proof}[Proof of Theorem \ref{T: dwork trace}]
Let $B_\kappa(t)$ be the infinite dimensional matrix of $[\bar \alpha_a]_\kappa$ with respect to the basis $\c B := \{ \w{m} : m \geq 0 \}$. Write $B_\kappa(t) = \sum_{n \geq 0} b_n t^n$, where $b_n$ is an infinite matrix with entries in $\bb C_p$. Define $F_{B_\kappa} := (b_{q n - m} )_{(n,m)}$ where $n, m \geq 0$, and we set $b_{qn-m} := 0$ if $qn-m < 0$, the zero matrix. As described prior to \cite[Lemma 2.3]{MR3239170}, the matrix of $\beta_\kappa$ with respect to $\c B$ is $F_{B_\kappa}$. By \cite[Lemma 4.1]{Wan-p-adic-representation}, the Dwork trace formula gives
\begin{align*}
(q^m - 1)Tr( \beta_\kappa^m ) &= (q^m - 1) Tr( F_{B_\kappa}^m) \\
&= \sum_{\substack{ \bar t \in \bb F_{q^m}^* \\ \hat t = \text{Teich}(\bar t)}}  Tr( B_\kappa(\hat t^{q^{m-1}}) \cdots B_\kappa(\hat t^q) B_\kappa(\hat t) ) \\
&= \sum_{\substack{ \bar t \in \bb F_{q^m}^* \\ \hat t = \text{Teich}(\bar t)}}  Tr( [ \bar \alpha_{\hat t} ]_{\kappa}^m \mid S_{\hat t}(\ep) )
\end{align*}
It now follows from (\ref{E: local factor}) that 
\[
L(Sym^{\infty, \kappa} Kl, T) = det(1 - \beta_\kappa T \mid S(b', \ep))^{\delta_q}.
\]
(See the argument succeeding \cite[Equation 8]{MR3249829} for details.)
\end{proof}

\subsection{Proof of Lemma \ref{L: partial beta}}\label{SS: proof of commute of partial}

The result follows from a limit using finite symmetric powers. Let $k$ be a positive integer. Define the map
\[
Sym^k \bar \alpha_m: Sym^k_{L(b')} H_t^1(b', b) \rightarrow Sym^k_{L(b'/p^m)} H_{t^{p^m}}^1( b' / p^m, b)
\]
Define the $\text{length} (w^{(m)}) := m$. For $\xi \in S(b', \ep)$, define $\text{length}(\xi)$ as the supremum of the lengths of the individual terms in the series defining $\xi$. In most cases, the length of $\xi$ will be infinite. Set $w_0 := 1$ and $w_1 := \pi t/x$, so that $\{w_0, w_1\}$ is a basis of $H_t^1(b', b)$, and $\{ w_0^{k-m} w_1^m : 0\leq m \leq k\}$ is a basis of $Sym^k_{L(b')} H_t^1(b', b)$ over $L(b')$. As the basis is finite, $\{ w_0^{k-m} w_1^{(m)} : 0\leq m \leq k\}$ is also a basis of $Sym^k_{L(b')} H_t^1(b', b)$ over $L(b')$, where $w_1^{(m)} := k^{\u m} w_1^m$. Define $S^{(k)}(b', \ep) :=  \{ \xi \in S(b', \ep)  \mid \text{length}(\xi) \leq k \}$. We may identify $S^{(k)}(b', \ep)$ with $Sym^{k} H_t^1(b', b)$ using the map
\[
w^{(m)} \longmapsto w_0^{k-m} w_1^{(m)}. 
\]
Let $k_n$ be a sequence of positive integers tending to infinity such that $p$-adically $k_n \rightarrow \kappa$. For each $n \in \bb Z_{\geq 0}$, define the approximation map $[ \bar \alpha_a ]_{(\kappa; n)}: S(b', \ep) \rightarrow S(b'/q, \ep)$ by
\[
[ \bar \alpha_a ]_{(\kappa; n)}( w^{(m)} ) := 
\begin{cases}
[ \bar \alpha_a ]_{k_n}( w^{(m)} ) & \text{if } m \leq k_n \\
0 & \text{otherwise.}
\end{cases} 
\]
Using the identification with the $k_n$-symmetric power, we have
\begin{align*}
[\bar \alpha_a]_{(\kappa; n)}(w^{(m)}) &=  (\Upsilon_q \circ \bar \alpha_a(1))^{k_n - m} ( \Upsilon_q \circ \bar \alpha_a \frac{\pi t}{x} )^m \\
&\cong \left( Sym^{k_n} \bar \alpha_a  \right)(w_0^{k_n - m} w_1^{(m)}),
\end{align*}
and thus $[\bar \alpha_a]_{(\kappa; n)} \cong  Sym^{k_n} \bar \alpha_a$.

Next, an analogous argument to that in \cite[Lemma 2.2]{MR3249829} demonstrates that  $\lim_{n \rightarrow \infty} [ \bar \alpha_a ]_{(\kappa; n)} = [\bar \alpha_a]_\kappa$ as maps from $S(b', \ep) \rightarrow S(b'/q, \ep)$. Consequently, if we define $\beta_{(\kappa; n)} := \psi_t^a  \circ [ \bar \alpha_a]_{(\kappa; n)}$ then as operators on $S(b', \ep)$, 
\begin{equation}\label{E: beta limit}
\lim_{n \rightarrow \infty} \beta_{(\kappa; n)} = \beta_{\kappa}.
\end{equation}
Lastly, define $\partial_{(\kappa; n)}$ on $S(b', \ep)$ as follows. For $0 \leq m \leq k_n-1$,
\begin{align*}
0 \leq m \leq k_n-1&: \qquad \partial_{(\kappa; n)}(t^r \w{m}) := r t^r \w{m} +  t^r \w{m+1} + m (k_n - m +1) \pi^2 t^{r+1} \w{m-1}, \\
m = k_n&: \qquad \partial_{(\kappa; n)}(t^r \w{k}) := r t^r \w{k} + k_n \pi^2 t^{r+1} \w{k-1}, \\
m > k_n&: \qquad \partial_{(\kappa; n)}(t^r \w{k}) := 0.
\end{align*}
Similarly, define $\tilde \partial_{(\kappa; n)}$ on $Sym^{k_n}_{L(b')} H_t^1(b', b)$ as follows. For $(\xi_1, \ldots, \xi_{k_n}) \in H_t^1(b', b)^{\oplus k_n}$, define
\[
\tilde \partial_{(\kappa; n)}(\xi_1 \cdots \xi_{k_n}) := \sum_{i=1}^{k_n} \xi_1 \cdots \hat \xi_i \cdots \xi_{k_n} \partial(\xi_i),
\]
where $\partial$ was defined by (\ref{E: diff of partial}). Then $\partial_{(\kappa; n)} \cong \tilde \partial_{(\kappa; n)}$ through the identification of the spaces $S^{(k_n)}(b', \ep)$ and $Sym^{k_n}_{L(b')} H_t^1(b', b)$. Further, again using the identification, 
\begin{equation}\label{E: est partial beta}
q \partial_{(\kappa; n)} \circ \beta_{(\kappa; n)} = \beta_{(\kappa; n)} \circ \partial_{(\kappa; n)}.
\end{equation}
Lastly, observe that with the topology of coefficient-wise convergence on $S(b', \ep)$, $\partial_{(\kappa; n)} \rightarrow \partial_\kappa$, which follows by considering the case $0 \leq m \leq k_n-1$:
\[
| (\partial_{(\kappa; n)} - \partial_\kappa)(t^r\w{m}) | = |  m ( k_n - \kappa) \pi^2 t^{r+1} \w{m-1} |,
\]
which tends to zero as $n \rightarrow \infty$. Lemma \ref{L: partial beta} now follows by taking the limit of (\ref{E: est partial beta}).

\bibliographystyle{amsplain}
\bibliography{../References/References.bib}

\end{document}